\documentclass{amsart}%
\usepackage{amsmath}
\usepackage{graphicx}
\usepackage{amsfonts}
\usepackage{amssymb}%
\providecommand{\U}[1]{\protect\rule{.1in}{.1in}}
\providecommand{\U}[1]{\protect\rule{.1in}{.1in}}
\providecommand{\U}[1]{\protect\rule{.1in}{.1in}}
\providecommand{\U}[1]{\protect\rule{.1in}{.1in}}
\providecommand{\U}[1]{\protect\rule{.1in}{.1in}}
\providecommand{\U}[1]{\protect\rule{.1in}{.1in}}
\providecommand{\U}[1]{\protect\rule{.1in}{.1in}}
\newtheorem{theorem}{Theorem}
\theoremstyle{plain}
\newtheorem{acknowledgement}{Acknowledgement}

\newtheorem{definition}{Definition}
\newtheorem{example}{Example}

\newtheorem{lemma}{Lemma}

\newtheorem{remark}{Remark}

\numberwithin{equation}{section}
\begin{document}
\title[Isoperimetric weights and generalized uncertainty inequalities]{Isoperimetric weights and generalized uncertainty inequalities in metric
measure spaces}
\author{Joaquim Mart\'{\i}n}
\address{Department of Mathematics\\
Universitat Aut\`{o}noma de Barcelona\\
jmartin@mat.uab.cat}
\author{Mario Milman}
\address{Instituto Argentino de Matematica\\
mario.milman@gmail.com\\
https://sites.google.com/site/mariomilman/}
\thanks{The first named author was partially supported by Grants MTM2013-44304-P,
MTM2013-40985-P, 2014SGR289}
\thanks{The second named author was partially supported by a grant from the Simons
Foundation (\#207929 to Mario Milman).}
\dedicatory{In memory of Nigel Kalton}\subjclass{ }
\maketitle

\begin{abstract}
We extend the recent $L^{1}$ uncertainty inequalities obtained in
\cite{daltre} to the metric setting. For this purpose we introduce a new class
of \ weights, named *isoperimetric weights*, for which the growth of the
measure of their level sets $\mu(\left\{  w\leq r\right\}  )$ can be
controlled by $rI(r),$ where $I$ is the isoperimetric profile of the ambient
metric space. We use isoperimetric weights, new *localized Poincar\'{e}
inequalities*, and interpolation, to prove $L^{p},1\leq p<\infty,$ uncertainty
inequalities on metric measure spaces. We give an alternate characterization
of the class of isoperimetric weights in terms of Marcinkiewicz spaces, which
combined with the sharp Sobolev inequalities of \cite{mamiadv}, and
interpolation of weighted norm inequalities, give new uncertainty inequalities
in the context of rearrangement invariant spaces.

\end{abstract}

\section{Introduction}

In a recent paper, Dall'ara-Trevisan \cite{daltre} extended the classical
uncertainty inequality (cf. \cite{weyl})\footnote{A detailed survey of the
uncertainty inequality and many related inequalities can be found in
\cite{folland}.}
\begin{equation}
\left\Vert f\right\Vert _{L^{2}(\mathbb{R}^{n})}^{2}\leq4n^{-2}\left\Vert
\left\vert \nabla f\right\vert \right\Vert _{L^{2}(\mathbb{R}^{n})}\left\Vert
\left\vert x\right\vert f\right\Vert _{L^{2}(\mathbb{R}^{n})},\text{ }f\in
C_{0}^{\infty}(\mathbb{R}^{n}), \label{intro1}%
\end{equation}
to a large class of homogenous spaces $M$ for a (Lie or finitely generated)
group $G$ such that the isotropy subgroups are compact and, furthermore, $M$
is endowed with an invariant measure $\mu$, an invariant distance $d,$ and an
invariant gradient which is compatible with $d$. To describe the weights
considered in \cite{daltre} let us observe that for each $r>0,$ the elements
of $B(r)$, the class of balls of radius $r$ in $M,$ have equal measure and,
consequently, one can consider the class of weights $w:M\rightarrow
\mathbb{R}^{+}$ that satisfy%
\begin{equation}
\mu(\{w\leq r\})\leq \Upsilon_{M}(r):=\mu(B(r)). \label{condbola}%
\end{equation}
In this setting, Dall'ara-Trevisan \cite{daltre} show that, for all weights
that satisfy (\ref{condbola}), there exists $c>0$ such that for all
$p\in\lbrack1,\infty),$%
\begin{equation}
\left\Vert f\right\Vert _{L^{p}(M,\mu)}^{2}\leq cp\left\Vert \left\vert \nabla
f\right\vert \right\Vert _{L^{p}(M,\mu)}\left\Vert wf\right\Vert _{L^{p}%
(M,\mu)}, \label{intro1i}%
\end{equation}
for all smooth functions $f$ satisfying suitably prescribed
cancellations\footnote{For example, if $M$ is compact, a natural normalization
condition is $\int fd\mu=0,$ and in the non-compact case (cf. \cite{daltre})
it is natural to require that the functions have compact support. In this
paper $Lip_{0}(\Omega)\,,$ will always denote the set of Lip functions with
compact support.}.

A new feature of the result is the fact that it is crucially valid for $p=1.$
Indeed, the inequalities for $L^{p},$ $p>1,$ follow from the $L^{1}$ case by a
familiar argument using the chain rule (cf. \cite{daltre}, \cite{sal}).
Furthermore, as in the classical theory of Sobolev inequalities of Maz'ya (cf.
\cite{mazya}, \cite{maz1}), the $L^{1}$ uncertainty inequalities are naturally
connected with isoperimetry. For example, when $M$ is compact, the
\textquotedblleft weak isoperimetric inequality\textquotedblright\ property
used in \cite{daltre} asserts the existence of a constant $C>0,$ such that for
all Borel sets $A$\thinspace and $E,$ with $\mu(A)\leq\mu(M)/2,$ and
$\mu(E)\leq \Upsilon_{M}(r),$ we have\footnote{For more on this we refer to
\cite{daltre} and Section \ref{nota3} below.}%
\begin{equation}
\mu(A\cap E)\leq Cr\mu^{+}(A), \label{condbola1}%
\end{equation}
where $\mu^{+}$ is a suitable notion of perimeter\footnote{Roughly speaking
\textquotedblleft$\mu^{+}(A)=\left\Vert |\nabla(\chi_{A})|\right\Vert _{L^{1}%
}".$} (cf. \cite{daltre}, and also \cite{coulh}). The proof of (\ref{intro1i})
in \cite{daltre} uses (\ref{condbola}) and (\ref{condbola1}) combined with
Poincar\'{e}'s inequality; furthermore, the group structure associated with
$M$ also plays a r\^{o}le.

The purpose of this paper is to extend (\ref{intro1i}) to the more general
context of metric measure spaces. In particular, we will eliminate the
dependence on any type of group structure as well as the requirement that the
measure of a ball depends only on its radius. In particular, our uncertainty
inequalities are also valid in the Gaussian setting. We are also able to
extend (\ref{intro1i}) to rearrangement invariant norms and establish a
principle that allows the transference of uncertainty inequalities between
different geometries, under the assumption that the underlying isoperimetric
profiles can be compared pointwise.

To explain in somewhat more detail the motivation behind the results, and the
methods we shall develop in this paper, we need to introduce some notation.

Let $(\Omega,\mu,d)$ be a connected metric measure space\footnote{We shall
list further assumptions on $(\Omega,\mu,d)$ as needed.}, such that
$\mu(K)<\infty,$ for compact sets $K\subset\Omega.$ The modulus of the
gradient of a Lip function $f$ is defined by
\[
\left\vert \nabla f(x)\right\vert =\limsup_{d(x,y)\rightarrow0}\frac
{|f(x)-f(y)|}{d(x,y)}.
\]
The perimeter or Minkowski content of a Borel set $A\subset\Omega,$ is defined
by
\[
\mu^{+}(A)=\liminf_{h\rightarrow0}\frac{\mu\left(  A_{h}\right)  -\mu\left(
A\right)  }{h},
\]
where $A_{h}=\left\{  x\in\Omega:d(x,A)<h\right\}  .$ The isoperimetric
profile\footnote{While isoperimetric profiles are very hard to compute
exactly, most of the estimates in this paper hold true if replace $I$ by a
lower bound estimator function, usually refered to as *an isoperimetric
estimator* (cf. \cite{mamiadv}).} $I:=I_{(\Omega,\mu,d)}$ associated with
$(\Omega,\mu,d),$ is the function $I:[0,\mu(\Omega))\rightarrow\mathbb{R}%
_{+},$ defined by%
\[
I(t)=\inf_{\mu(A)=t}\{\mu^{+}(A)\}.
\]
In what follows we shall only consider metric measure spaces such that $I$ is
concave, continuous, and zero at zero. Moreover, in the case of finite measure
spaces, i.e. when $\mu(\Omega)<\infty,$ we shall also assume that $I$ is
symmetric around $\mu(\Omega)/2.$ In particular, in this case $I$ will be
increasing on $(0,\mu(\Omega)/2),$ and decreasing on $(\mu(\Omega
)/2,\mu(\Omega)).$ Moreover, if $\mu(\Omega)=\infty,$ we assume that $I$ is increasing.

Our main focus will be on the validity of $L^{1}$ inequalities of the
form\footnote{Below we will also develop methods to treat uncertantity
inequalities for rather general rearrangement invariant norms.}%
\[
\left\Vert f\right\Vert _{L^{1}(\Omega,\mu)}^{2}\leq c\left\Vert \left\vert
\nabla f\right\vert \right\Vert _{L^{1}(\Omega,\mu)}\left\Vert wf\right\Vert
_{L^{1}(\Omega,\mu)},
\]
for all smooth functions $f$ that satisfy suitably prescribed cancellations.
As in \cite{daltre} one of our main tools will be local Poincar\'{e}
inequalities, but in our case, they are formulated using the isoperimetric
profile, and are valid for arbitrary measurable sets, rather than balls. For
example, if $\mu(\Omega)<\infty,$ we localize the usual Poincar\'{e}
inequality (cf. \cite{bobhou}, \cite{mamiadv}, and the references therein),
\begin{equation}
\int_{{\Omega}}\left\vert f-m(f)\right\vert d\mu\leq\frac{\mu(\Omega)}%
{2I(\mu(\Omega)/2)}\int_{{\Omega}}\left\vert \nabla f\right\vert d\mu,
\label{pr1}%
\end{equation}
as follows: For all $f\in Lip(\Omega)$ and for all measurable $A\subset\Omega$
we have (cf. Theorem \ref{ivan} below)%
\begin{equation}
\int_{{A}}\left\vert f-m(f)\right\vert d\mu\leq\frac{\min\{\mu(A),\mu
(\Omega)/2\}}{I(\min\{\mu(A),\mu(\Omega)/2\})}\int_{{\Omega}}\left\vert \nabla
f\right\vert d\mu, \label{pr1loc}%
\end{equation}
where $m(f)$ is a median of $f$ (cf. Section \ref{prel}).

A key role in our analysis is played by a class of weights, which we call
\textbf{isoperimetric weights}. A positive measurable function $w:\Omega
\rightarrow\mathbb{R}_{+}$ will be called an isoperimetric weight if there
exists a constant $C=C(w)$ such that%
\begin{equation}
\frac{\min\{\mu(\left\{  w\leq r\right\}  ),\mu(\Omega)/2\}}{I(\min
\{\mu(\left\{  w\leq r\right\}  ),\mu(\Omega)/2\})}\leq Cr,\text{ \ \ }r>0.
\label{pesoiso}%
\end{equation}
In particular, when $\mu(\Omega)=\infty,$ the condition (\ref{pesoiso}) takes
the simpler form\footnote{To understand the reason why condition
(\ref{pesoiso}) is slightly more complicated when $\mu(\Omega)<\infty,$ note
that if $\mu(\{w\leq r\})=1,$ then, since $I(1)=0,$ $I(\mu(\{w\leq r\}))=0$
and (\ref{pesoiso1}) has no meaning. A comparable phenomenon occurs with
condition (\ref{condbola}), which has no meaning when $r>$diameter of $M.$}%
\begin{equation}
\mu(\{w\leq r\})\leq CrI(\mu(\{w\leq r\})),\text{ \ \ }r>0. \label{pesoiso1}%
\end{equation}
As a consequence, the growth of the measure of the level sets of isoperimetric
weights is controlled by the isoperimetric profile associated with the
geometry. To get some insight on the difference between (\ref{pesoiso1}) and
(\ref{condbola}) we shall now briefly compare them in the context of
$\mathbb{R}^{n}.$ The classical weight used for Euclidean uncertainty
inequalities (cf. (\ref{intro1}) above) is $w(x)=\left\vert x\right\vert $.
For this weight, both conditions, (\ref{pesoiso1}) and (\ref{condbola}), are
satisfied, but the calculations needed for their verifications are different.
Indeed, let $\mu_{R^{n}}$ denote the Lebesgue measure on $\mathbb{R}^{n},$ and
let $w(x)=\left\vert x\right\vert ,$ then we have%
\[
\mu_{\mathbb{R}^{n}}(\{w\leq r\})=\mu_{\mathbb{R}^{n}}(\{\left\vert
x\right\vert \leq r\})=\beta_{n}r^{n},
\]
where $\beta_{n}$ is the measure of the unit ball; on the other hand, since
$I_{\mathbb{R}^{n}}(r)=n\left(  \beta_{n}\right)  ^{1/n}r^{(1-1/n)},$ we also
have
\begin{align*}
rI_{\mathbb{R}^{n}}(\mu(\{w  &  \leq r\}))=rn\left(  \beta_{n}\right)
^{1/n}(\beta_{n}r^{n})^{(1-1/n)}\\
&  =n\beta_{n}r^{n}.
\end{align*}
Thus, $w(x)=\left\vert x\right\vert $ satisfies both (\ref{pesoiso1}) and
(\ref{condbola}). In fact, more generally, for geometries that satisfy the
assumptions of \cite{daltre} and, moreover, have concave isoperimetric
profiles, we will show that if a weight $w$ satisfies (\ref{condbola}) then it
is an isoperimetric weight in our sense (cf. Section \ref{exam}, Theorem
\ref{treviteo}).

For isoperimetric weights we will show (cf. Theorem \ref{Teounl1} below) that
there exists a constant $c=c(w)$ such that, for all suitably normalized
Lipschitz functions $f,$ the following uncertainty inequality holds
\begin{equation}
\left\Vert f\right\Vert _{L^{1}(\Omega,\mu)}^{2}\leq c\left\Vert \left\vert
\nabla f\right\vert \right\Vert _{L^{1}(\Omega,\mu)}\left\Vert wf\right\Vert
_{L^{1}(\Omega,\mu)}. \label{arkansa}%
\end{equation}
Moreover, a weak converse holds. Namely, if (\ref{arkansa}) holds for a given
weight $w,$ then it is easy to see that the growth of the measure of the level
sets of $w$ must be controlled in some fashion by their corresponding
perimeters. More precisely, we have (cf. Remark \ref{xx} below),%
\[
\mu\left(  \{w\leq r\}\right)  \leq cr\mu^{+}\left(  \{w\leq r\}\right)  .
\]

From a technical point of view, the class of isoperimetric weights is useful
for our development in this paper since these weights are directly related to
the local Poincar\'{e} inequalities described above (cf. (\ref{pr1loc})). In
fact, with these tools at hand, combined with interpolation\footnote{Actually
we do not use the abstract theory of interpolation but simply the technique of
splitting a function into suitable pieces using its level sets.}, we are able
to adapt the main argument of \cite{daltre} to prove $L^{1}$ uncertainty
inequalities in our setting. As it turns out, there is still a different
characterization of the class of isoperimetric weights through the use of
rearrangements. Indeed, we will show that isoperimetric weights are functions
that belong to a Marcinkiewicz space whose fundamental function behaves
essentially like $\frac{t}{I(t)}$. We then observe that, in view of a
classical inequality of Hardy-Littlewood, the usual self improvements of
Sobolev inequalities can be formulated as weighted norm inequalities, where
the weights are precisely the isoperimetric weights! At this point we use
interpolation to derive new uncertainty inequalities for rearrangement
invariant norms.

Let $\Phi(t):=\Phi_{I}(t)=\frac{\min\{t,\mu(\Omega)/2\}}{I(\min\{t,\mu
(\Omega)/2\})},$ $t\in(0,\mu(\Omega)),$ then, since we assume that $I(t)$ is
concave, the function $\Phi$ is non-decreasing. It can be readily seen (cf.
Lemma \ref{lmar} in Section \ref{isoweigt} below) that $w$ is an isoperimetric
weight if and only if $W:=\frac{1}{w}$ belongs to the Marcinkiewicz space
$M(\Phi)=M(\Phi)(\Omega,\mu),$ of functions on $\Omega$ such that
\[
\left\Vert f\right\Vert _{M(\Phi)}=\sup_{t>0}t\Phi(\mu\{\left\vert
f\right\vert >t\})<\infty.
\]
In terms of rearrangements (cf. \cite{stewei} and Section \ref{isoweigt}
below) we can also write%
\[
\left\Vert f\right\Vert _{M(\Phi)}=\sup_{t>0}f_{\mu}^{\ast}(t)\Phi(t).
\]
In other words, $w$ is an isoperimetric weight if and only $W:=\frac{1}{w}$
satisfies%
\begin{equation}
\left\Vert W\right\Vert _{M(\Phi)}=\sup_{t>0}W_{\mu}^{\ast}(t)\Phi
(t)=\sup_{t>0}W_{\mu}^{\ast}(t)\frac{t}{I(t)}<\infty. \label{marcin}%
\end{equation}
To set the stage for the more general developments we shall present in Section
\ref{unrii}, let us briefly develop, in the more familiar Euclidean setting,
the connection of uncertainty inequalities with sharp Sobolev inequalities and
explain the r\^{o}le of the Marcinkiewicz space $M(\Phi)$. The key idea here
is that the classical Gagliardo-Nirenberg inequality\footnote{Here and in what
follows we let
\[
Lip_{0}(\Omega)=\{f\in Lip(\Omega):f\text{ \ has compact support\}}.
\]
}
\[
\left\Vert f\right\Vert _{L^{\frac{n}{n-1}}(\mathbb{R}^{n})}\leq\frac
{1}{n\left(  \beta_{n}\right)  ^{1/n}}\left\Vert \left\vert \nabla
f\right\vert \right\Vert _{L^{1}(\mathbb{R}^{n})},\text{ }f\in Lip_{0}%
(\mathbb{R}^{n}),
\]
self improves\footnote{Note that%
\[
\left\Vert f\right\Vert _{L^{\frac{n}{n-1}}}\leq\frac{n}{n-1}\left\Vert
f\right\Vert _{L^{\frac{n}{n-1},1}}.
\]
} to (cf. \cite{tal})%
\begin{equation}
\left\Vert f\right\Vert _{L^{\frac{n}{n-1},1}(\mathbb{R}^{n})}:=\int%
_{0}^{\infty}f^{\ast}(s)s^{1-1/n}\frac{ds}{s}\leq\frac{n^{\prime}}{\left(
\beta_{n}\right)  ^{1/n}}\left\Vert \left\vert \nabla f\right\vert \right\Vert
_{L^{1}(\mathbb{R}^{n})},\text{ }f\in Lip_{0}(\mathbb{R}^{n}). \label{marcin1}%
\end{equation}
This self improvement can be re-interpreted as an $L^{1}(\Omega,\mu)$ weighted
inequality. Indeed, suppose that $W\in M(\Phi)(\mathbb{R}^{n},\mu
_{\mathbb{R}^{n}}),$ where $\Phi(t)=t^{1/n},$ and $d\mu_{\mathbb{R}^{n}%
}(x)=dx$ is the Lebesgue measure. Then, for $f\in Lip_{0}(\mathbb{R}^{n}),$ we
have,%
\begin{align}
\left\Vert fW\right\Vert _{L^{1}(R^{n})}  &  =\int_{\mathbb{R}^{n}}\left\vert
f(x)\right\vert W(x)dx\nonumber\\
&  \leq\int_{R^{n}}f^{\ast}(t)W^{\ast}(t)dt\text{ (by the Hardy-Littlewood
inequality)}\nonumber\\
&  \leq\left\Vert W\right\Vert _{M(\Phi)}\int_{\mathbb{R}^{n}}f^{\ast
}(t)t^{1-1/n}\frac{dt}{t}\text{ \ (recall (\ref{marcin}))}\nonumber\\
&  \leq\frac{n^{\prime}}{\left(  \beta_{n}\right)  ^{1/n}}\left\Vert
W\right\Vert _{M(\Phi)}\left\Vert \left\vert \nabla f\right\vert \right\Vert
_{L^{1}(R^{n})},\text{ }f\in Lip_{0}(\mathbb{R}^{n})\text{ (by (\ref{marcin1}%
))}. \label{marcin3}%
\end{align}
It follows from H\"{o}lder's inequality\footnote{This argument was provided by
the referee, previously we had indicated a proof by interpolation.} that
\[
\int_{\mathbb{R}^{n}}\left\vert f(x)\right\vert dx\leq c_{n}\left\Vert
W\right\Vert _{M(\Phi)}^{1/2}\left\Vert \left\vert \nabla f\right\vert
\right\Vert _{L^{1}(\mathbb{R}^{n})}^{1/2}\left\Vert fw\right\Vert
_{L^{1}(\mathbb{R}^{n})}^{1/2}.
\]
The weighted norm inequality (\ref{marcin3}) appears already in \cite{faris},
and corresponds to one of the end points of the Strichartz
inequalities\footnote{Note that (\ref{marcin3}) and (\ref{marcinp}) are Hardy
type inequalities. In general, the case $p=1,$ seems to be new.} (cf.
\cite[Sec II, Theorem 3.6, page 1049]{strichartz}),%
\begin{equation}
\left\Vert fW\right\Vert _{L^{p}(\mathbb{R}^{n})}\leq c_{n}(p)\left\Vert
W\right\Vert _{L^{n,\infty}}\left\Vert \left\vert \nabla f\right\vert
\right\Vert _{L^{p}(\mathbb{R}^{n})},\text{ \ }1<p<n,\text{ \ }f\in
Lip_{0}(\mathbb{R}^{n}). \label{marcinp}%
\end{equation}
In fact, taking as a starting point the sharp Sobolev inequality of
Hardy-Littlewood-O'Neil\footnote{Comparing methods, in \cite{faris},
(\ref{marcinp}) is proved directly and then (\ref{sobo}) is obtained as a
corollary} \cite{oneil}%
\begin{equation}
\left\Vert f\right\Vert _{L^{\bar{p},p}(\mathbb{R}^{n})}\leq c_{n}\left\Vert
\left\vert \nabla f\right\vert \right\Vert _{L^{p}(\mathbb{R}^{n})},\ \frac
{1}{\bar{p}}=\frac{1}{p}-\frac{1}{n},\ 1\leq p<n, \label{sobo}%
\end{equation}
and following the argument that led us to (\ref{marcin3}) gives a proof of
(\ref{marcinp}) (cf. Section \ref{seccstrich} below). More generally, it is
perhaps a new observation that the corresponding weighted norm inequalities
implied by the sharp Sobolev inequalities of \cite{mamiadv} extend the
Strichartz \cite{strichartz} and Faris \cite{faris} inequalities to the
setting of r.i. norms (cf. Section \ref{seccstrich} below). Finally, let us
remark that this discussion gives another proof of the Euclidean space version
of the uncertainty inequality of Dall'ara-Trevisan \cite{daltre}. Indeed, for
$w(x)=\left\vert x\right\vert ,$ then $W(x)=\left\vert x\right\vert ^{-1},$
and\footnote{Here the notation $f\approx g$ indicates the existence of a
universal constant $c>0$ (independent of all parameters involved) such that
$(1/c)f\leq g\leq c\,f$.}
\[
c(W)=\left\Vert \left\vert x\right\vert ^{-1}\right\Vert _{M(\Phi)}%
\approx\left\Vert \left\vert x\right\vert ^{-1}\right\Vert _{L^{n,\infty
}(\mathbb{R}^{n})}=\beta^{1/n}.
\]

As a bonus, this approach to (\ref{intro1i}) allow us to also replace the
$L^{1}$ norms by rearrangement invariant norms. In fact, the argument can be
also adapted to deal Besov type conditions (cf. Section \ref{seccbesov}
below). The connection with Marcinkiewicz spaces makes it also easy to
actually construct isoperimetric weights for given geometries where we have a
lower bound on the corresponding isoperimetric profiles, as we show with
examples in Section \ref{exam} below. In particular, we show how to construct
isoperimetric weights for Gaussian or more generally log concave measures (cf.
Section \ref{logconcave}).

For perspective, the connection between (the classical) Sobolev inequalities
and Lorentz-Marcinkiewicz spaces $L^{p,\infty},$ has been known for a long
time, and already appears, albeit implicitly, in the work of Hardy-Littlewood,
and is already fully developed and exploited in the celebrated work of O'Neil
on convolution inequalities (cf. \cite{oneil}). It is in \cite{oneil} that one
finds explicitly the idea of using Marcinkiewicz spaces in order to treat
abstractly convolution with potentials of the form $w=f(d(x,x_{0})).$ In our
context, the \textquotedblleft good weights" \ belong to Marcinkiewicz spaces
whose very definition is given in terms of the underlying isoperimetric
profile. In the metric setting the use of weight functions of the form
$w=f(d(x,x_{0})),$ where $d$ is the underlying metric, is classical (cf.
\cite{oko} and the references therein). In particular, we remark that for
geometries where the measure of a ball is independent of the
radius\footnote{This condition fails for Gaussian measure (cf. \cite{forzani},
\cite[Proposition 5.1, page 52.]{urbina}).}, the functions of the form
$w(x)=d(x_{0},x),$ where $x_{0}$ is a fixed element of $\Omega,$ trivially
satisfy the condition (\ref{condbola}) with equality.

We should also mention that the characterization of isoperimetric weights
using Marcinkiewicz spaces also readily leads to a transference result for
uncertainty inequalities which we formulate in Section \ref{secc:transfer} below.

Finally, and without any claim to completeness, we give a sample of recent
references that treat uncertainty inequalities in different contexts and with
different levels of generality, where the reader may find further references
to the large literature in this field (cf. \cite{ciattia}, \cite{ciatti},
\cite{martini}, \cite{oko}, \cite{ricci}), we also should mention the
classical papers by Fefferman \cite{fe} and Beckner \cite{bk}, \cite{bk1}.

\section{Preliminaries\label{prel}}

In this section we establish some further notation and background information
and we provide more details about isoperimetric weights. Let $(\Omega,\mu,d)$
be a metric measure space as described in the Introduction.

\subsection{Medians}

In this subsection we assume that $\mu(\Omega)<\infty.$

\begin{definition}
Let $f:\Omega\rightarrow\mathbb{R}$ be an integrable function. We say that
$m(f)$ is a median of $f$ if%
\[
\mu\{f\geq m(f)\}\geq\mu(\Omega)/2;\text{ and }\mu\{f\leq m(f)\}\geq\mu
(\Omega)/2.
\]

\end{definition}

For later use we record the following elementary estimate of the median, and
provide the easy proof for the sake of completeness,%
\begin{equation}
\left\vert m(f)\right\vert \leq\frac{2}{\mu(\Omega)}\int_{\Omega}\left\vert
f\right\vert d\mu. \label{lestima}%
\end{equation}

\begin{proof}
We use Chebyshev's inequality and the definition of median as follows. On the
one hand,
\begin{align*}
m(f)  &  =m(f)\mu\{f\geq m(f)\}\frac{1}{\mu\{f\geq m(f)\}}\\
&  \leq\frac{1}{\mu\{f\geq m(f)\}}\int_{\{f\geq m(f)\}}fd\mu\\
&  \leq\frac{2}{\mu(\Omega)}\int_{\Omega}\left|  f\right|  d\mu.
\end{align*}
On the other hand, we similarly have%
\begin{align*}
m(f)  &  =m(f)\mu\{f\leq m(f)\}\frac{1}{\mu\{f\leq m(f)\}}\\
&  =\frac{m(f)}{\mu\{f\leq m(f)\}}\int_{\{f\leq m(f)\}}d\mu\\
&  \geq\frac{1}{\mu\{f\leq m(f)\}}\int_{\{f\leq m(f)\}}fd\mu.
\end{align*}
Hence,%
\begin{align*}
-m(f)  &  \leq\frac{1}{\mu\{f\leq m(f)\}}\int_{\{f\leq m(f)\}}-fd\mu\\
&  \leq\frac{1}{\mu\{f\leq m(f)\}}\int_{\Omega}\left|  f\right|  d\mu\\
&  \leq\frac{2}{\mu(\Omega)}\int_{\Omega}\left|  f\right|  d\mu.
\end{align*}
Combining estimates (\ref{lestima}) follows.
\end{proof}

\subsection{Rearrangement invariant spaces\label{secc:ri}}

Let $u:\Omega\rightarrow\mathbb{R},$ be a measurable function. The
\textbf{distribution function} of $u$ is given by
\[
\mu_{u}(t)=\mu\{x\in{\Omega}:\left\vert u(x)\right\vert >t\}\text{
\ \ \ \ }(t\geq0).
\]
The \textbf{decreasing rearrangement} of a function $u$ is the
right-continuous non-increasing function from $[0,\mu(\Omega))$ into
$\mathbb{R}^{+}$ which is equimeasurable with $u.$ It can be defined by the
formula
\[
u_{\mu}^{\ast}(s)=\inf\{t\geq0:\mu_{u}(t)\leq s\},\text{ \ }s\in\lbrack
0,\mu(\Omega)),
\]
and satisfies
\[
\mu_{u}(t)=\mu\{x\in{\Omega}:\left\vert u(x)\right\vert >t\}=m\left\{
s\in\lbrack0,\mu(\Omega)):u_{\mu}^{\ast}(s)>t\right\}  \text{, \ }t\geq0,
\]
(where $m$ denotes the Lebesgue measure on $[0,\mu(\Omega)).$ It follows from
the definition that
\begin{equation}
\left(  u+v\right)  _{\mu}^{\ast}(s)\leq u_{\mu}^{\ast}(s/2)+v_{\mu}^{\ast
}(s/2). \label{a1}%
\end{equation}

The maximal average $u_{\mu}^{\ast\ast}(t)$ is defined by
\[
u_{\mu}^{\ast\ast}(t)=\frac{1}{t}\int_{0}^{t}u_{\mu}^{\ast}(s)ds=\frac{1}%
{t}\sup\left\{  \int_{E}\left\vert u(s)\right\vert d\mu:\mu(E)=t\right\}
,\text{ }t>0.
\]
The operation $u\rightarrow u_{\mu}^{\ast\ast}$ is sub-additive, i.e.
\begin{equation}
\left(  u+v\right)  _{\mu}^{\ast\ast}(s)\leq u_{\mu}^{\ast\ast}(s)+v_{\mu
}^{\ast\ast}(s), \label{a2}%
\end{equation}
and moreover,
\begin{equation}
\int_{0}^{t}(uv)_{\mu}^{\ast}(s)ds\leq\int_{0}^{t}u_{\mu}^{\ast}(s)v_{\mu
}^{\ast}(s)ds,\text{ }t>0. \label{a3}%
\end{equation}

On occasion, when rearrangements are taken with respect to the Lebesgue
measure or when the measure is clear from the context, we may omit the measure
and simply write $u^{\ast}$ and $u^{\ast\ast}$, etc.

We now recall briefly the basic definitions and conventions we use from the
theory of rearrangement-invariant (r.i.) spaces and refer the reader to
\cite{BS} for a complete treatment. We say that a Banach function
space\footnote{We use the definition of Banach function space that one can
find in \cite{BS} which, in particular, assumes that the spaces have the Fatou
property.} $X=X({\Omega}),$ on $({\Omega},d,\mu),$ is a
rearrangement-invariant (r.i.) space, if $g\in X$ implies that all $\mu
-$measurable functions $f$ with the same rearrangement with respect to the
measure $\mu$, i.e. such that $f_{\mu}^{\ast}=g_{\mu}^{\ast},$ also belong to
$X,$ and, moreover, $\Vert f\Vert_{X}=\Vert g\Vert_{X}$.

When dealing with r.i. spaces we will always assume that $({\Omega},d,\mu)$ is
resonant in the sense of \cite[Definition 2.3, pag 45]{BS}.

For any r.i. space $X({\Omega})$ we have%

\[
L^{\infty}(\Omega)\cap L^{1}(\Omega)\subset X(\Omega)\subset L^{1}%
(\Omega)+L^{\infty}(\Omega),
\]
with continuous embeddings. In particular, if $\mu$ is finite, then%
\[
L^{\infty}(\Omega)\subset X(\Omega)\subset L^{1}(\Omega).
\]

An r.i. space $X({\Omega})$ can be represented by a r.i. space on the interval
$(0,\mu(\Omega)),$ with Lebesgue measure, $\bar{X}=\bar{X}(0,\mu(\Omega)),$
such that (see \cite[Theorem 4.10 and subsequent remarks]{BS})
\[
\Vert f\Vert_{X}=\Vert f_{\mu}^{\ast}\Vert_{\bar{X}},
\]
for every $f\in X.$ Typical examples of r.i. spaces are the $L^{p}$-spaces,
Lorentz spaces, Marcinkiewicz spaces and Orlicz spaces.

A useful property of r.i. spaces states that if
\[
\int_{0}^{t}\left\vert f\right\vert _{\mu}^{\ast}(s)ds\leq\int_{0}%
^{t}\left\vert g\right\vert _{\mu}^{\ast}(s)ds,
\]
\ holds for all\ $t>0,$ and $X$ is a r.i. space, then,%
\begin{equation}
\left\Vert f\right\Vert _{X}\leq\left\Vert g\right\Vert _{X}. \label{hardyine}%
\end{equation}

\subsection{Isoperimetric weights\label{isoweigt}}

We give a formal discussion of the notion of isoperimetric weight and its
characterization in terms of Marcinkiewicz spaces.

\begin{definition}
We will say that a locally integrable function $w:\Omega\rightarrow
\mathbb{R}_{+}$ is an \textbf{isoperimetric weight}, if $w>0$ a.e.$,$ and
there exists a constant $C:=C(w)>0$ such that%
\begin{equation}
\frac{\min\{\mu(\left\{  w\leq r\right\}  ),\mu(\Omega)/2\}}{I(\min
\{\mu(\left\{  w\leq r\right\}  ),\mu(\Omega)/2\})}\leq Cr,\text{ \ \ }r>0.
\label{unouno}%
\end{equation}

\end{definition}

It is easy to see that (\ref{unouno}) is equivalent to%
\begin{equation}
\frac{\min\{\mu(\left\{  w<r\right\}  ),\mu(\Omega)/2\}}{I(\min\{\mu(\left\{
w<r\right\}  ),\mu(\Omega)/2\})}\leq Cr,\ \ r>0. \label{unounoprima}%
\end{equation}

In what follows we write%
\[
C(w):=\sup_{r>0}\frac{1}{r}\frac{\min\{\mu(\left\{  w\leq r\right\}
),\mu(\Omega)/2\}}{I(\min\{\mu(\left\{  w\leq r\right\}  ),\mu(\Omega)/2\})}.
\]

Let $\Phi(t)$ $=\frac{\min\{t,\mu(\Omega)/2\}}{I(\min\{t,\mu(\Omega)/2\})},$
$t\in(0,\mu(\Omega)).$ The \textbf{Marcinkiewicz} $M\left(  \Phi\right)
(\Omega)$ is defined by the condition%
\[
\left\Vert f\right\Vert _{M(\Phi)}=\sup_{t>0}t\Phi(\mu_{f}(t))<\infty.
\]
Since $f_{\mu}^{\ast}$ and $\mu_{f}$ are generalized inverses of each other a
simple argument (cf. \cite{stewei}) shows that $f\in M\left(  \Phi\right)
(\Omega)$ if and only if%
\[
\left\Vert f\right\Vert _{M(\Phi)}=\sup_{t>0}f_{\mu}^{\ast}(t)\Phi(t)<\infty.
\]

The previous discussion leads to the following

\begin{lemma}
\label{lmar}$w$ is an isoperimetric weight if and only if $W:=\frac{1}{w}\in
M\left(  \Phi\right)  (\Omega).$
\end{lemma}

\begin{proof}
From (\ref{unounoprima}) and the fact that $\mu(\left\{  w<r\right\}
)=\mu_{W}(\frac{1}{r}),$ we have%
\begin{align*}
C(w)  &  =\sup_{r>0}\frac{1}{r}\frac{\min\{\mu(\left\{  w<r\right\}
),\mu(\Omega)/2\}}{I(\min\{\mu(\left\{  w<r\right\}  ),\mu(\Omega)/2\})}\\
&  =\sup_{r>0}\frac{1}{r}\frac{\min\{\mu_{W}(\frac{1}{r}),\mu(\Omega
)/2\}}{I(\min\{\mu_{W}(\frac{1}{r}),\mu(\Omega)/2\})}\\
&  =\sup_{r>0}r\Phi\mu_{W}(r)\\
&  =\left\Vert W\right\Vert _{M(\Phi)}.
\end{align*}

\end{proof}

\begin{remark}
\label{xx}In some sense we don't have too many choices of weights in order for
uncertainty inequalities to be true. For example, suppose that $\mu
(\Omega)=\infty,$ and $w$ is a weight such that
\begin{equation}
\left\Vert f\right\Vert _{L^{1}(\Omega)}^{2}\leq c\left\Vert \left\vert \nabla
f\right\vert \right\Vert _{L^{1}(\Omega)}\left\Vert wf\right\Vert
_{L^{1}(\Omega)} \label{propiedades}%
\end{equation}
holds for all $f\in Lip_{0}(\Omega).$ Suppose that $\mu(\{w\leq t\})<\infty$,
then,
\begin{equation}
\mu(\{w\leq t\})\leq ct\mu^{+}(\{w\leq t\}). \label{propiedades1}%
\end{equation}

\end{remark}

\begin{proof}
We can select $f_{n}\in Lip_{0}(\Omega)$ such that%
\[
\left\Vert \left\vert \nabla f_{n}\right\vert \right\Vert _{L^{1}(\Omega
)}\rightarrow\mu^{+}(\{w\leq t\})
\]
while%
\[
\left\Vert wf_{n}\right\Vert _{L^{1}(\Omega)}\rightarrow\int_{\{w\leq
t\}}wd\mu
\]
and%
\[
\left\Vert f_{n}\right\Vert _{L^{1}(\Omega)}^{2}\rightarrow\mu^{2}(\{w\leq
t\}).
\]
Inserting this information back in (\ref{propiedades}), we have%
\begin{align*}
\mu^{2}(\{w  &  \leq t\})\leq c\mu^{+}(\{w\leq t\})\int_{\{w\leq t\}}wd\mu\\
&  \leq c\mu^{+}(\{w\leq t\})t\mu(\{w\leq t\}),
\end{align*}
and (\ref{propiedades1}) follows.
\end{proof}

We now prove the localized Poincar\'{e} inequality described in the Introduction.

\begin{theorem}
\label{ivan}Suppose that $\mu(\Omega)<\infty.$ Then, for all $f\in
Lip(\Omega),$ and for all measurable $A\subset\Omega,$ we have%
\begin{equation}
\int_{{A}}\left\vert f-m(f)\right\vert d\mu\leq\frac{\min\{\mu(A),\mu
(\Omega)/2\}}{I(\min\{\mu(A),\mu(\Omega)/2\})}\int_{{\Omega}}\left\vert \nabla
f\right\vert d\mu. \label{poinlocal}%
\end{equation}

\end{theorem}

Before giving the simple proof we recall the following important consequence
of the co-area formula and the definition of isoperimetric profile (cf.
\cite{bobhou}, \cite{mamiadv}).

\begin{theorem}
(i) Suppose that $\mu(\Omega)<\infty$ (resp. $\mu(\Omega)=\infty).$ Then, for
all $f\in Lip(\Omega)$ (resp. for all $f\in Lip_{0}(\Omega)),$ we have
\begin{equation}
\int_{-\infty}^{\infty}I(\mu(\{f>s\}))ds\leq\int_{{\Omega}}\left\vert \nabla
f\right\vert d\mu. \label{ledoux}%
\end{equation}

\end{theorem}

\begin{proof}
(\textbf{of Theorem \ref{ivan})} Let $f\in Lip(\Omega),$ and let
$A\subset\Omega$ be a measurable set. We compute
\begin{align}
\int_{{A}}\left\vert f-m(f)\right\vert d\mu &  =\int_{{A\cap}\left\{  f\geq
m(f)\right\}  }\left(  f-m(f)\right)  d\mu+\int_{{A\cap}\left\{
f<m(f)\right\}  }\left(  m(f)-f\right)  d\mu\nonumber\\
&  =\int_{{0}}^{\infty}\mu(\{f-m(f)>s\}\cap A\})ds+\int_{{0}}^{\infty}%
\mu(\{m(f)-f\geq s\}\cap A\})ds\nonumber\\
&  =\int_{m(f)}^{\infty}\mu(\left\{  f>s\right\}  \cap A\})ds+\int_{-\infty
}^{m(f)}\mu(\left\{  f\leq s\right\}  \cap A\})ds.\nonumber\\
&  =(I)+(II). \label{uno}%
\end{align}
We estimate $(I)$. Suppose that $s>m(f).$ We claim that%
\begin{equation}
\mu(\{f>s\}\cap A\})\leq\min\{\mu(A),\mu(\Omega)/2\}. \label{adicional}%
\end{equation}
It is plain that (\ref{adicional}) will follow if can show that $\mu
(\{f>s\})\leq\mu(\Omega)/2.$ Suppose, to the contrary, that $\mu
(\{f>s\})>\mu(\Omega)/2.$ Then, since $\{f>s\}$ and $\{f\leq m(f)\}$ are
disjoint sets, $\mu(\{f>s\})+\mu(\{f\leq m(f)\})\leq\mu(\Omega).$ It follows
that $\mu(\{f\leq m(f)\})<\mu(\Omega)/2,$ which is impossible since $m(f)$ is
a median of $f.$

From (\ref{adicional}), and the fact that $t/I(t)$ increases, we get%
\begin{equation}
\mu(\{f>s\}\cap A\})\leq I(\mu(\{f>s\}\cap A\}))\frac{\min\{\mu(A),\mu
(\Omega)/2\}}{I(\min\{\mu(A),\mu(\Omega)/2\})}. \label{adicional1}%
\end{equation}
Since $I$ is increasing on $\left(  0,\mu(\Omega)/2\right]  ,$ and $\mu\left(
\{f>s\}\right)  \leq\mu(\Omega)/2$, we see that
\[
I(\mu(\{f>s\}\cap A\}))\leq I(\mu(\{f>s\})).
\]
Updating (\ref{adicional1}) we have%
\[
\mu(\{f>s\}\cap A\})\leq I(\mu(\{f>s\}))\frac{\min\{\mu(A),\mu(\Omega
)/2\}}{I(\min\{\mu(A),\mu(\Omega)/2\})}.
\]
Integrating we obtain%
\begin{equation}
(I)\leq\frac{\min\{\mu(A),\mu(\Omega)/2\}}{I(\min\{\mu(A),\mu(\Omega
)/2\})}\int_{m(f)}^{\infty}I(\mu(\{f>s\}))ds. \label{uno1}%
\end{equation}
In a similar way we can estimate $(II)$%
\begin{equation}
(II)\leq\frac{\min\{\mu(A),\mu(\Omega)/2\}}{I(\min\{\mu(A),\mu(\Omega
)/2\})}\int_{-\infty}^{m(f)}I(\mu(\left\{  f\leq s\right\}  ))ds. \label{uno2}%
\end{equation}
Inserting the estimates (\ref{uno1}) and (\ref{uno2}) in (\ref{uno}) we obtain%
\begin{equation}
\int_{{A}}\left\vert f-m(f)\right\vert d\mu\leq\frac{\min\{\mu(A),\mu
(\Omega)/2\}}{I(\min\{\mu(A),\mu(\Omega)/2\})}\left(  \int_{m(f)}^{\infty
}I(\mu(\{f>s\}))ds+\int_{-\infty}^{m(f)}I(\mu(\left\{  f\leq s\right\}
))ds\right)  . \label{update}%
\end{equation}
We now show that the integrals inside the parenthesis can be combined. Indeed,
when $s<m(f)$ we have $\mu(\left\{  f\leq s\right\}  )<\mu(\Omega)/2,$
therefore, by the symmetry of $I$ around the point $\mu(\Omega)/2,$ we find
that
\[
I(\mu(\left\{  f\leq s\right\}  ))=I(\mu(\Omega)-\mu(\left\{  f\leq s\right\}
))=I(\mu(\Omega\diagdown\left\{  f\leq s\right\}  ))=I(\mu(\left\{
f>s\right\}  )).
\]
Whence,%
\[
\int_{-\infty}^{m(f)}I(\mu(\left\{  f\leq s\right\}  ))ds=\int_{-\infty
}^{m(f)}I(\mu(\left\{  f>s\right\}  ))ds.
\]
Inserting the last equality in (\ref{update}) yields%
\begin{align*}
\int_{{A}}\left\vert f-m(f)\right\vert d\mu &  \leq\frac{\min\{\mu
(A),\mu(\Omega)/2\}}{I(\min\{\mu(A),\mu(\Omega)/2\})}\int_{-\infty}^{\infty
}I(\mu(\{f>s\}))ds\\
&  \leq\frac{\min\{\mu(A),\mu(\Omega)/2\}}{I(\min\{\mu(A),\mu(\Omega
)/2\})}\int_{{\Omega}}\left\vert \nabla f\right\vert d\mu\text{ \ (by
(\ref{ledoux}))},
\end{align*}
as we wished to show.
\end{proof}

\begin{remark}
\label{poincar}Note that (\ref{poinlocal}) reduces to (\ref{pr1}) when
$A=\Omega.$ Moreover, since%
\[
\frac{1}{2}\int_{{\Omega}}\left\vert f-\frac{1}{\mu({\Omega})}\int_{\Omega
}fd\mu\right\vert d\mu\leq\int_{{\Omega}}\left\vert f-m(f)\right\vert d\mu,
\]
we also get
\[
\int_{{\Omega}}\left\vert f-\frac{1}{\mu({\Omega})}\int_{\Omega}%
fd\mu\right\vert d\mu\leq\frac{\mu({\Omega})}{I(\mu({\Omega})/2)}\int%
_{{\Omega}}\left\vert \nabla f\right\vert d\mu.
\]

\end{remark}

\section{$L^{1}$ Uncertainty Inequalities via local Poincar\'{e} inequalities}

Let $(\Omega,\mu,d)$ be a metric measure space. In this section we prove our
main result concerning $L^{1}$ uncertainty inequalities.

\begin{theorem}
\label{Teounl1}Let $w$ be an isoperimetric weight, and let $\alpha>0.$ Suppose
that $\mu(\Omega)<\infty$ (resp. $\mu(\Omega)=\infty$), then, for all $f\in
Lip(\Omega),$ with $m(f)=0$ or $\int_{\Omega}fd\mu=0$ (resp. for all $f\in
Lip_{0}(\Omega))$, we have
\begin{equation}
\left\Vert f\right\Vert _{1}\leq2C(w)r\left\Vert \left\vert \nabla
f\right\vert \right\Vert _{1}+2r^{-\alpha}\left\Vert w^{\alpha}f\right\Vert
_{1},\text{ for all }r>0. \label{unl1}%
\end{equation}

\end{theorem}

\begin{proof}
\textbf{Case of finite measure.}

Suppose that $f\in Lip(\Omega).$ We consider each normalization separately.

(i) Suppose $m(f)=0.$ Then, for all $r>0,$ we have%
\begin{align*}
\int_{\Omega}\left\vert f\right\vert d\mu &  =\int_{\{w\leq r\}}\left\vert
f\right\vert d\mu+\int_{\{r<w\}}\left\vert f\right\vert d\mu\\
&  =\int_{\{w\leq r\}}\left\vert f-m(f)\right\vert d\mu+\int_{\{w>r\}}%
\left\vert f\right\vert d\mu\\
&  \leq\frac{\min\{\mu(\{w\leq r\}),\mu(\Omega)/2\}}{I(\min\{\mu(\{w\leq
r\}),\mu(\Omega)/2\})}\int_{{\Omega}}\left\vert \nabla f\right\vert d\mu
+\int_{\{w>r\}}\left\vert f\right\vert d\mu\text{ \ (by (\ref{poinlocal}))}\\
&  \leq C(w)r\int_{{\Omega}}\left\vert \nabla f\right\vert d\mu+\int%
_{\{w>r\}}\left\vert f\right\vert d\mu\text{ (since }w\text{ is an
isoperimetric weight)}\\
&  =C(w)r\int_{{\Omega}}\left\vert \nabla f\right\vert d\mu+\int%
_{\{w>r\}}\left(  \frac{w}{w}\right)  ^{\alpha}\left\vert f\right\vert
d\mu\text{ }\\
&  \leq C(w)r\int_{{\Omega}}\left\vert \nabla f\right\vert d\mu+r^{-\alpha
}\int_{\Omega}w^{\alpha}\left\vert f\right\vert d\mu,\text{ }%
\end{align*}
as desired.

(ii) Suppose that $\int_{\Omega}fd\mu=0.$ Let $r>0,$ and write%
\begin{align*}
\int_{\Omega}\left\vert f\right\vert d\mu &  =\int_{\{w\leq r\}}\left\vert
f\right\vert d\mu+\int_{\{w>r\}}\left\vert f\right\vert d\mu\\
&  \leq\int_{\{w\leq r\}}\left\vert f-m(f)\right\vert d\mu+\left\vert
m(f)\right\vert \mu(\{w\leq r\})+\int_{\{w>r\}}\left\vert f\right\vert d\mu\\
&  \leq\int_{\{w\leq r\}}\left\vert f-m(f)\right\vert d\mu+\left\vert
\int_{\{w\leq r\}}m(f)d\mu-\int_{\{w\leq r\}}fd\mu-\int_{\{w>r\}}%
fd\mu\right\vert +\int_{\{w>r\}}\left\vert f\right\vert d\mu\\
&  \leq2\int_{\{w\leq r\}}\left\vert f-m(f)\right\vert d\mu+2\int_{\{w\leq
r\}}\left\vert f\right\vert d\mu\\
&  =2A(r)+2C(r).
\end{align*}

Since the terms $A(r)$ and $C(r)$ were estimated in the proof of $(i)$ above,
it follows that
\[
\int_{\Omega}\left\vert f\right\vert d\mu\leq2C(w)r\int_{\Omega}\left\vert
\nabla f\right\vert d\mu+2r^{-\alpha}\int_{\Omega}w^{\alpha}\left\vert
f\right\vert d\mu
\]
as we wished to show.

\textbf{Case of infinite measure. }$f\in Lip_{0}(\Omega).$ For all $r>0$ we
write%
\begin{align}
\left\Vert f\right\Vert _{1}  &  =\int_{\{w\leq r\}}\left\vert f\right\vert
d\mu+\int_{\{w>r\}}\left\vert f\right\vert d\mu\nonumber\\
&  =(I)+(II). \label{inserta}%
\end{align}
The term $(II)$ can be estimated exactly as in the previous case. It remains
to estimate%
\[
(I)=\int_{{0}}^{\infty}\mu(\{\left\vert f\right\vert >s\}\cap\{w\leq r\})ds.
\]
The integrand can be estimated as follows,%
\begin{align*}
\mu(\{\left\vert f\right\vert  &  >s\}\cap\{w\leq r\})\leq I(\mu(\{\left\vert
f\right\vert >s\}\cap\{w\leq r\}))\frac{\mu(\{w\leq r\})}{I(\mu(\{w\leq
r\}))}\text{ (since }\frac{s}{I(s)}\text{ increases)}\\
&  \leq C(w)rI(\mu(\{\left\vert f\right\vert >s\}\cap\{w\leq r\}))\text{
(since }w\text{ is an isoperimetric weight)}\\
&  \leq C(w)rI(\mu(\{\left\vert f\right\vert >s\}))\text{ (since }I\text{
increases).}%
\end{align*}
Consequently,%
\begin{align*}
\int_{\{w\leq r\}}\left\vert f\right\vert d\mu &  \leq C(w)r\int_{{0}}%
^{\infty}I(\mu(\{\left\vert f\right\vert >s\}))ds\text{ }\\
&  \leq C(w)r\int_{{\Omega}}\left\vert \nabla\left\vert f\right\vert
\right\vert d\mu\text{ \ (by (\ref{ledoux}))}\\
&  \leq C(w)r\int_{{\Omega}}\left\vert \nabla f\right\vert d\mu.
\end{align*}

Inserting the estimates that we have obtained for $(I)$ and $(II)$ into
(\ref{inserta}) gives the desired result.
\end{proof}

\begin{remark}
\label{reml1}Selecting the value $r=\left(  \frac{\left\Vert w^{\alpha
}f\right\Vert _{1}}{2C(w)\left\Vert \left\vert \nabla f\right\vert \right\Vert
_{1}}\right)  ^{\frac{1}{1+\alpha}}$ to compute (\ref{unl1}) balances the two
terms and we obtain the multiplicative inequality%
\[
\left\Vert f\right\Vert _{1}\leq\left(  2C(w)\right)  ^{\frac{\alpha}%
{^{\alpha+1}}}\left\Vert \left\vert \nabla f\right\vert \right\Vert
_{1}^{\frac{\alpha}{\alpha+1}}\left\Vert w^{\alpha}f\right\Vert _{1}^{\frac
{1}{\alpha+1}},
\]
for all $f\in Lip(\Omega)$ such that $m(f)=0,$ or $\int_{\Omega}fd\mu=0$ if
$\mu(\Omega)<\infty$ (or for all $f\in Lip_{0}(\Omega),$ if $\mu
(\Omega)=\infty).$
\end{remark}

Following closely the chain rule argument used in \cite[Section 6.5]{daltre}
we now show that Theorem \ref{Teounl1} implies the corresponding $L^{p}$
version of itself. More precisely, we have

\begin{theorem}
\label{teounclp}Let $w$ be an isoperimetric weight, let $p>1,$ and let
$\alpha>0.$ Suppose that $\mu(\Omega)<\infty$ (resp. $\mu(\Omega)=\infty$),
then, there is a constant $D=D(w,p)$ such that, for all $f\in Lip(\Omega)$
with $m(f)=0,$ or $\int_{\Omega}fd\mu=0$ (resp. for all $f\in Lip_{0}%
(\Omega))$, we have
\[
\left\Vert f\right\Vert _{p}\leq Dr\left\Vert \left\vert \nabla f\right\vert
\right\Vert _{p}+2r^{-\alpha}\left\Vert w^{\alpha}\left\vert f\right\vert
\right\Vert _{p},\text{ \ \ \ }r>0.
\]

\end{theorem}

\begin{proof}
The main technical difficulty to prove the Theorem is that, as we attempt to
apply Theorem \ref{Teounl1} by replacing $f$ with powers of $f,$ we may loose
the required normalizations in the process. Let us first record the following
well known elementary version of the chain rule for power functions, which is
valid in the metric setting,
\[
\left\vert \nabla\left\vert f\right\vert ^{p}(x)\right\vert \leq2p\left\vert
f(x)\right\vert ^{p-1}\left\vert \nabla\left\vert f\right\vert (x)\right\vert
\leq2p\left\vert f(x)\right\vert ^{p-1}\left\vert \nabla f(x)\right\vert
,\ f\in Lip(\Omega),
\]
and%
\[
\left\vert \nabla\left(  f\left\vert f\right\vert ^{p-1}\right)
(x)\right\vert \leq2p\left\vert f(x)\right\vert ^{p-1}\left\vert \nabla
f(x)\right\vert ,\ f\in Lip(\Omega).
\]
Applying H\"{o}lder's inequality we find%
\begin{equation}
\left\Vert \left\vert \nabla\left\vert f\right\vert ^{p}\right\vert
\right\Vert _{1}\leq2p\left\Vert \left\vert \nabla f\right\vert \left\vert
f\right\vert ^{p-1}\right\Vert _{1}\leq2p\left\Vert \left\vert \nabla
f\right\vert \right\Vert _{p}\left\Vert f\right\Vert _{p}^{p-1}, \label{orla}%
\end{equation}
and
\begin{equation}
\left\Vert \nabla\left(  f\left\vert f\right\vert ^{p-1}\right)  \right\Vert
_{1}\leq2p\left\Vert \left\vert \nabla f\right\vert \right\Vert _{p}\left\Vert
f\right\Vert _{p}^{p-1}. \label{orla0}%
\end{equation}
Let us also note here that H\"{o}lder's inequality gives us,%
\begin{equation}
\left\Vert w^{\alpha}\left\vert f\right\vert ^{p}\right\Vert _{1}=\left\Vert
w^{\alpha}\left\vert f\right\vert \left\vert f\right\vert ^{p-1}\right\Vert
_{1}\leq\left\Vert w^{\alpha}\left\vert f\right\vert \right\Vert
_{p}\left\Vert f\right\Vert _{p}^{p-1}. \label{orla1}%
\end{equation}
We now divide the proof in several cases. The easiest case to deal with the
normalizations issue is when $\mu(\Omega)=\infty.$ Indeed, if $f\in
Lip_{0}(\Omega),$ we also have $\left\vert f\right\vert ^{p}\in Lip_{0}%
(\Omega).$ Thus, by (\ref{unl1}) we have,%
\[
\left\Vert \left\vert f\right\vert ^{p}\right\Vert _{1}\leq2C(w)r\left\Vert
\left\vert \nabla\left\vert f\right\vert ^{p}\right\vert \right\Vert
_{1}+r^{-\alpha}\left\Vert w^{\alpha}\left\vert f\right\vert ^{p}\right\Vert
_{1},\text{ for all }r>0.
\]
Then, from (\ref{orla}) and (\ref{orla1}), we get,%
\[
\left\Vert f\right\Vert _{p}^{p}\leq4C(w)p\left\Vert \left\vert \nabla
f\right\vert \right\Vert _{p}\left\Vert f\right\Vert _{p}^{p-1}+\left\Vert
w^{\alpha}\left\vert f\right\vert \right\Vert _{p}\left\Vert f\right\Vert
_{p}^{p-1},\text{ for all }r>0,
\]
and the desired result follows.

Suppose that $\mu(\Omega)<\infty.$ Let $f\in Lip\left(  \Omega\right)  $. It
is not difficult to verify that $m(f)\left\vert m(f)\right\vert ^{p-1}$ is a
median of $f\left\vert f\right\vert ^{p-1}$ (cf. \cite{daltre}). To take
advantage of this fact we estimate $\left\Vert \left\vert f\right\vert
^{p}\right\Vert _{1}$ as follows%
\begin{align*}
\int_{\Omega}\left\vert f\right\vert ^{p}d\mu &  =\int_{\{w\leq r\}}\left\vert
f\right\vert ^{p}d\mu+\int_{\{r<w\}}\left\vert f\right\vert ^{p}d\mu\\
&  \leq\int_{\{w\leq r\}}\left\vert f\left\vert f\right\vert ^{p-1}%
-m(f)\left\vert m(f)\right\vert ^{p-1}\right\vert d\mu+\left\vert
\int_{\{w\leq r\}}m(f)\left\vert m(f)\right\vert ^{p-1}d\mu\right\vert \\
&  +\int_{\{w>r\}}\left\vert f\right\vert ^{p}d\mu\\
&  =A(r)+B(r)+C(r).
\end{align*}
$A(r)$ can be estimated using the local Poincar\'{e} inequality, the fact that
$w$ is an isoperimetric weight and (\ref{orla}):%
\begin{align*}
A(r)  &  \leq C(w)r\left\Vert \left\vert \nabla\left(  f\left\vert
f\right\vert ^{p-1}\right)  \right\vert \right\Vert _{1}\\
&  \leq2pC(w)r\left\Vert \left\vert \nabla f\right\vert \right\Vert
_{p}\left\Vert f\right\Vert _{p}^{p-1}\text{ (by (\ref{orla0}))}.
\end{align*}
The term $C(r)$ can be readily estimated in a familiar fashion:%
\begin{align*}
C(r)  &  \leq r^{-\alpha}\left\Vert w^{\alpha}\left\vert f\right\vert
^{p}\right\Vert _{1}\\
&  \leq\left\Vert w^{\alpha}\left\vert f\right\vert \right\Vert _{p}\left\Vert
f\right\Vert _{p}^{p-1}\text{ (by (\ref{orla1})).}%
\end{align*}

It remains to estimate $B(r),$ and for this purpose we shall consider two
cases. If $m(f)=0,$ then $B(r)=0$ and we are done. Suppose now that
$\int_{\Omega}fd\mu=0.$ Then we can write%
\begin{align*}
B(r)  &  =\left\vert \int_{\{w\leq r\}}m(f)d\mu-\int_{\{w\leq r\}}fd\mu
-\int_{\{w>r\}}fd\mu\right\vert \left\vert m(f)\right\vert ^{p-1}\\
&  \leq\left\vert m(f)\right\vert ^{p-1}\int_{\{w\leq r\}}\left\vert
m(f)-f\right\vert d\mu+\left\vert m(f)\right\vert ^{p-1}\int_{\{w>r\}}%
\left\vert f\right\vert d\mu\\
&  =B_{1}(r)+B_{2}(r).
\end{align*}
To estimate $B_{1}(r)$ we consider each of its factors separately. Using the
local Poincar\'{e} inequality, the definition of isoperimetric weight and
H\"{o}lder's inequality we see that%
\begin{align*}
\int_{\{w\leq r\}}\left\vert m(f)-f\right\vert d\mu &  \leq c(w)r\int%
_{{\Omega}}\left\vert \nabla f\right\vert d\mu\\
&  \leq c(w)r\left(  \int_{{\Omega}}\left\vert \nabla f\right\vert ^{p}%
d\mu\right)  ^{1/p}\mu(\Omega)^{1/p^{\prime}}.
\end{align*}
To estimate $\left\vert m(f)\right\vert ^{p-1}$ we observe that
$m(f)\left\vert m(f)\right\vert ^{p-2}$ is a median of $f\left\vert
f\right\vert ^{p-2}.$ Therefore, by (\ref{lestima}) and H\"{o}lder's
inequality,%
\begin{align*}
\left\vert m(f)\right\vert ^{p-1}  &  =\left\vert m(f)\left\vert
m(f)\right\vert ^{p-2}\right\vert \\
&  \leq\frac{2}{\mu(\Omega)}\int_{{\Omega}}\left\vert f\left\vert f\right\vert
^{p-2}\right\vert d\mu\\
&  =\frac{2}{\mu(\Omega)}\int_{{\Omega}}\left\vert f\right\vert ^{p-1}d\mu\\
&  \leq\frac{2}{\mu(\Omega)}\left(  \int_{{\Omega}}\left\vert f\right\vert
^{p}d\mu\right)  ^{1/p^{\prime}}\mu({\Omega)}^{1/p}.
\end{align*}
It follows that%
\begin{align*}
B_{1}(r)  &  \leq C(w)r\left(  \int_{{\Omega}}\left\vert \nabla f\right\vert
^{p}d\mu\right)  ^{1/p}\mu(\Omega)^{1/p^{\prime}}\frac{2}{\mu(\Omega)}\left(
\int_{{\Omega}}\left\vert f\right\vert ^{p}d\mu\right)  ^{1/p^{\prime}}%
\mu({\Omega)}^{1/p}\\
&  =2C(w)r\left\Vert \nabla f\right\Vert _{p}\left\Vert f\right\Vert
_{p}^{p-1}.
\end{align*}
Similarly, using the estimate for $\left\vert m(f)\right\vert ^{p-1}$ obtained
above and H\"{o}lder's inequality, we see that
\begin{align*}
B_{2}(r)  &  =\left\vert m(f)\right\vert ^{p-1}\int_{\{w>r\}}\left\vert
f\right\vert d\mu\\
&  \leq\frac{2}{\mu(\Omega)}\left(  \int_{{\Omega}}\left\vert f\right\vert
^{p}d\mu\right)  ^{1/p^{\prime}}\mu({\Omega)}^{1/p}\left(  \int_{\{w>r\}}%
\left\vert f\right\vert ^{p}d\mu\right)  ^{1/p}\mu(\Omega)^{1/p^{\prime}}\\
&  \leq2\left\Vert f\right\Vert _{p}^{p-1}r^{-\alpha}\left\Vert w^{\alpha
}f\right\Vert _{p}.
\end{align*}
Therefore,%
\[
B(r)\leq2C(w)r\left\Vert \nabla f\right\Vert _{p}\left\Vert f\right\Vert
_{p}^{p-1}+2\left\Vert f\right\Vert _{p}^{p-1}r^{-\alpha}\left\Vert w^{\alpha
}f\right\Vert _{p}.
\]
Combining estimates we thus arrive at%
\[
\left\Vert f\right\Vert _{p}^{p}\leq C(w)(2p+1)r\left\Vert \left\vert \nabla
f\right\vert \right\Vert _{p}\left\Vert f\right\Vert _{p}^{p-1}+2r^{-\alpha
}\left\Vert w^{\alpha}\left\vert f\right\vert \right\Vert _{p}\left\Vert
f\right\Vert _{p}^{p-1},
\]
as desired.
\end{proof}

\begin{remark}
If we select $r=\left(  \frac{\left\Vert w^{\alpha}f\right\Vert _{p}%
}{D\left\Vert \left\vert \nabla f\right\vert \right\Vert _{p}}\right)
^{\frac{1}{1+\alpha}}$ to compute (\ref{unl1}), then we obtain the
multiplicative inequality%
\[
\left\Vert f\right\Vert _{p}\leq D^{\frac{p\alpha p}{^{p\alpha+p}}}\left\Vert
\left\vert \nabla f\right\vert \right\Vert _{p}^{\frac{\alpha}{\alpha+1}%
}\left\Vert w^{\alpha}f\right\Vert _{p}^{\frac{1}{\alpha+1}},
\]
for all $f\in Lip(\Omega)$ such that $m(f)=0$ or $\int_{\Omega}fd\mu=0$ if
$\mu(\Omega)<\infty$ (or for all $f\in Lip_{0}(\Omega)$ if $\mu(\Omega
)=\infty).$
\end{remark}

\section{Isoperimetric weights vs Dall'Ara-Trevisan weights\label{nota3}}

Dall'Ara-Trevisan \cite{daltre} proved versions of Theorems \ref{Teounl1} and
\ref{teounclp}, for homogeneous spaces\footnote{We refer to the Introduction
for the basic assumptions on $M,$ and \cite{daltre} for complete details.}
$M,$ and for weights $w:M\rightarrow\mathbb{R}^{+}$ that satisfy the growth condition%

\begin{equation}
\mu(\{w\leq r\})\leq \Upsilon_{M}(r):=\mu(B(r)). \label{intro2}%
\end{equation}
In this section we compare the weights in the Dall'Ara-Trevisan class with the
corresponding isoperimetric weights defined on $M.$ In preparation for this
task let us introduce some notation and recall useful information.

In this section we shall consider homogeneous spaces $M$ that satisfy the
assumptions of \cite{daltre} and, moreover, are metric measure spaces in the
sense of the present paper\footnote{In particular, we assume that the
isoperimetric profile $I_{(M,\mu,d)}:=I\ $\ satisfies the usual assumptions.}.
We shall simply refer to these spaces as *homogeneous metric measure spaces*.

We now recall the weak isoperimetric inequality used in \cite{daltre} (cf.
also \cite{coulh}).

\begin{theorem}
\label{weakiso}Let $M$ be an homogeneous space satisfying the assumptions of
\cite{daltre}. Then the following statements hold.

\begin{enumerate}
\item Suppose that $M$ is non-compact. Then, there exists $C>0$ such that for
all $r>0,$ and for all Borel sets $A\subset M$ such that $\mu(A)\leq
\Upsilon_{M}(r),$ we have
\begin{equation}
\mu(A)\leq Cr\mu^{+}(A). \label{tr0}%
\end{equation}

\item Suppose that $M$ is compact. Then,

(i) There exists $C>0$ such that for $r>0,$ and all Borel sets $E$ with
$\min\{\mu(E),\mu(E^{c})\}\leq\frac{\Upsilon_{M}(r)}{2},$ we have%
\[
\min\{\mu(E),\mu(E^{c})\}\leq Cr\mu^{+}(E).
\]
(ii) If $\mu(E)\leq \Upsilon_{M}(r),$ then, for all $r>0,$ and for all Borel
sets $A\subset M$ such that $\mu(A)\leq\mu(M)/2,$
\begin{equation}
\mu(A\cap E)\leq Cr\mu^{+}(A). \label{tr00}%
\end{equation}

\end{enumerate}
\end{theorem}

\begin{theorem}
\label{treviteo}Let $M$ be an homogeneous metric measure space. Then, the
class of isoperimetric weights contains the class of weights satisfying the
growth condition (\ref{intro2}).
\end{theorem}

The next result will be useful in the proof.

\begin{lemma}
(i) Let $M$ be an homogeneous metric measure space. Then, there exists $C>0$
such that, for all $r>0,$ with $\Upsilon_{M}(r)\leq\frac{\mu(M)}{2},$ it holds
that
\begin{equation}
\Upsilon_{M}(r)\leq CrI(\Upsilon_{M}(r)). \label{jubileo}%
\end{equation}
In particular, if $\Upsilon_{M}(r)\leq\frac{\mu(M)}{2},$ then, for any Borel
set $E$ such that $\mu(E)\leq \Upsilon_{M}(r),$%
\begin{equation}
\mu(E)\leq CrI(\mu(E)). \label{jubileo1}%
\end{equation}

\end{lemma}

\begin{proof}
Let $r>0$ be fixed$.$ Suppose that $\Upsilon_{M}(r)\leq\frac{\mu(M)}{2},$ and
let $A\subset M$ be a Borel set such that
\[
\mu(A)=\Upsilon_{M}(r).
\]
Using Theorem \ref{weakiso} we see that%
\[
\Upsilon_{M}(r)\leq rC\mu^{+}(A).
\]
Indeed, if $M$ is compact this follows directly from (\ref{tr00}), while in
the non compact case we can use (\ref{tr0}), with $E=A,$ to arrive to same
conclusion. Taking infimum we obtain,%
\begin{align*}
\Upsilon_{M}(r)  &  \leq rC\inf\{\mu^{+}(A):\mu(A)=\Upsilon_{M}(r)\}\\
&  =rCI(\Upsilon_{M}(r)),
\end{align*}
and therefore (\ref{jubileo}) holds.

Suppose now that $0<\mu(E)\leq \Upsilon_{M}(r)\leq\frac{\mu(M)}{2}.$ Then,
since $t/I(t)$ increases, we have%
\[
\frac{\mu(E)}{I(\mu(E))}\leq\frac{\Upsilon_{M}(r)}{I(\Upsilon_{M}(r))}\leq
Cr,
\]
as desired.
\end{proof}

We can now proceed with the proof of Theorem \ref{treviteo}.

\begin{proof}
We assume that $w$ is a weight such that $\mu(\{w\leq r\})\leq \Upsilon
_{M}(r).$ We are aiming to prove%
\begin{equation}
\min\left\{  \mu(\{w\leq r\}),\frac{\mu(M)}{2}\right\}  \leq CrI\left(
\min\{\mu(\{w\leq r\}),\frac{\mu(M)}{2}\}\right)  . \label{arcada}%
\end{equation}
\textbf{Case I}: Compact case. (i) Suppose that $\Upsilon_{M}(r)\leq\frac
{\mu(M)}{2}.$ Then by (\ref{jubileo1})%
\[
\mu(\{w\leq r\})\leq CrI(\{\mu(\{w\leq r\}).
\]
Since our assumptions on $w$ and $\Upsilon_{M}(r)$ force $\mu(\{w\leq
r\})\leq\frac{\mu(M)}{2},$ we see that (\ref{arcada}) holds.

(ii) Suppose that $\Upsilon_{M}(r)>\frac{\mu(M)}{2}.$ Suppose also that
$\mu(\{w\leq r\})>\frac{\mu(M)}{2}.$ Then, since $t/I(t)$ increases,%
\begin{align*}
\frac{\frac{\mu(M)}{2}}{I(\frac{\mu(M)}{2})}  &  \leq\frac{\mu(\{w\leq
r\})}{I(\mu(\{w\leq r\})}\\
&  \leq\frac{CrI(\mu(\{w\leq r\}))}{I(\mu(\{w\leq r\})}\text{ (by
(\ref{jubileo1}))}\\
&  =Cr.
\end{align*}
In other words, (\ref{arcada}) holds in this case as well.

(iii) It remains to consider the case $\Upsilon_{M}(r)>\frac{\mu(M)}{2}$,
$\mu(\{w\leq r\})\leq\frac{\mu(M)}{2}.$ By \cite{daltre} the function
$\Upsilon_{M}(s)$ is continuous and increasing. Let $r^{\prime}$ be such that
$\Upsilon_{M}(r^{\prime})=\frac{\mu(M)}{2}.$ Since $t/I(t)$ increases,%
\begin{align*}
\frac{\mu(\{w\leq r\})}{I(\mu(\{w\leq r\}))}  &  \leq\frac{\Upsilon
_{M}(r^{\prime})}{I(\Upsilon_{M}(r^{\prime}))}\\
&  \leq\frac{Cr^{\prime}I(\Upsilon_{M}(r^{\prime}))}{I(\Upsilon_{M}(r^{\prime
}))}\text{ (by (\ref{jubileo}))}\\
&  \leq Cr\text{ (since }r^{\prime}\leq r\text{).}%
\end{align*}
Therefore, (\ref{arcada}) holds in this case as well, concluding the proof of
Case I.

\textbf{Case II}: Non compact case. In this case we must have $\mu(M)=\infty$
(cf Remark \ref{compacto} below)$,$ then $\Upsilon_{M}(r)\leq\frac{\mu(M)}%
{2},$ and $\mu(\{w\leq r\})\leq\frac{\mu(M)}{2},$ therefore we see that
(\ref{arcada}) holds by (\ref{jubileo1}).
\end{proof}

\begin{remark}
\label{compacto}In \cite{daltre} the dichotomy for the normalization
conditions is given in terms of whether the space $M$ is compact or not. On
the other hand, the assumptions of \cite{daltre} force $\mu(M)<\infty,$ when
$M$ is compact and $\mu(M)=\infty,$ if $M$ is not compact. Indeed, in Section
4.1 of \cite{daltre} the authors show that for $M$ non compact, $\Upsilon
_{M}(r)<\infty,$ for all $r>0,$ and that $\Upsilon_{M}(r)+\Upsilon_{M}%
(s)\leq \Upsilon_{M}(r+s),$ for all $r,s>0.$ In particular, $n\Upsilon
_{M}(1)\leq \Upsilon_{M}(n),$ for all $n\in\mathbb{N},$ and therefore we have
$\Upsilon_{M}(n)\rightarrow\infty.$
\end{remark}

\section{Examples and applications\label{exam}}

Let $(\Omega,\mu,d)$ be a metric measure space. We will use the following
general scheme to construct isoperimetric weights in different settings. Let
$g:[0,\mu(\Omega)]\rightarrow\lbrack0,\mathbb{\infty)},$ be such that $g>0$
a.e., and
\begin{equation}
\sup_{0<t<\mu(\Omega)}g^{\ast}(t)\frac{\min\{t,\mu(\Omega)/2\}}{I(\min
\{t,\mu(\Omega)/2\}}<\infty, \label{marmar}%
\end{equation}
where the rearrangement is taken with respect to the Lebesgue measure on
$[0,\mu(\Omega)].$ It is known (cf. \cite[Corollary 7.8 pag. 86]{BS}) that
there exists a measure-preserving transformation $\sigma:\Omega\rightarrow
\lbrack0,\mu(\Omega)]$ such that $g^{\ast}\circ\sigma$ and $g^{\ast}$are
equimeasurable. In particular,
\[
\left(  g^{\ast}\circ\sigma\right)  _{\mu}^{\ast}(t)=g^{\ast}(t)\text{
\ \ \ }t\in\lbrack0,\mu(\Omega)].
\]
It follows that the function
\[
w(x)=\frac{1}{g^{\ast}(\sigma(x))},\text{ \ \ }x\in\Omega,
\]
is an isoperimetric weight. Indeed, by Lemma \ref{lmar}, $W(x)=\frac{1}%
{w(x)}=g^{\ast}(\sigma(x))\in M(\Phi)(\Omega).$ Consequently,%

\[
\left\Vert f\right\Vert _{p}\leq D^{\frac{p\alpha p}{^{p\alpha+p}}}\left\Vert
\left\vert \nabla f\right\vert \right\Vert _{p}^{\frac{\alpha}{\alpha+1}%
}\left\Vert \left(  \frac{1}{g^{\ast}\circ\sigma}\right)  ^{\alpha
}f\right\Vert _{p}^{\frac{1}{\alpha+1}},
\]
for all $f\in Lip(\Omega),$ such that $m(f)=0$ or $\int_{\Omega}fd\mu=0$ if
$\mu(\Omega)<\infty$ (or for all $f\in Lip_{0}(\Omega)$ if $\mu(\Omega
)=\infty).$

In the examples below we consider specific metric measure spaces and make
explicit calculations following the above scheme.

\subsection{Euclidean case}

\subsubsection{$\mathbb{R}^{n}$ with Lebesgue measure.}

The isoperimetric profile is given by%
\[
I_{\mathbb{R}^{n}}(r)=n\left(  \beta_{n}\right)  ^{1/n}r^{(1-1/n)}.
\]

\begin{theorem}
Let $g:[0,\infty)\rightarrow\lbrack0,\mathbb{\infty)}$ be such that $g>0$
a.e., and, moreover, suppose that
\[
\sup_{0<t<\infty}g^{\ast}(t)t^{1/n}<\infty.
\]
Let $w:\mathbb{R}^{n}\rightarrow\mathbb{R}^{+},$ be defined by
\[
w(x)=\frac{1}{g^{\ast}(\beta_{n}\left\vert x\right\vert ^{n})}.
\]
Then, $w$ is an isoperimetric weight.
\end{theorem}

\begin{proof}
Let $W=\frac{1}{w}.$ Then, $W(x)=g^{\ast}(\beta_{n}\left\vert x\right\vert
^{n})$ and $W^{\ast}(t)=g^{\ast}(t)$ (cf. \cite{Talen}). Consequently, $W\in
M(\Phi)\left(  \mathbb{R}^{n}\right)  ,$ and the result follows.
\end{proof}

\subsubsection{The closed upper half of Euclidean space $\mathbb{R}^{n}$}

For simplicity we assume that $n=2$.

Let $H_{2}=\{(x,y)\in\mathbb{R}^{2}:$ $y\geq0\}.$ By reflection across the
boundary of $H_{2},$ combined with the classical isoperimetric inequality in
$\mathbb{R}^{2},$ it follows that the corresponding isoperimetric profile,
$I_{H_{2}},$ is given by (cf. \cite{Talen})%
\[
I_{H_{2}}(t)=\beta_{2}^{1/2}t^{1/2}.
\]

\begin{theorem}
Let $g:[0,\infty)\rightarrow\lbrack0,\mathbb{\infty)}$ be such that $g>0$
a.e., and%
\[
\sup_{0<t<\infty}g^{\ast}(t)t^{1/2}<\infty.
\]
Let $k>0,$ and let $w:H_{2}\rightarrow\mathbb{R}^{+}$ be defined by
\[
w(x)=\frac{1}{g^{\ast}\left(  \frac{1}{k+1}B\left(  \frac{k+1}{2},\frac{1}%
{2}\right)  \left(  x^{2}+y^{2}\right)  ^{k/2+1}\right)  },
\]
where $B$ denotes the Euler beta function. Then, $w$ is an isoperimetric weight.
\end{theorem}

\begin{proof}
Let $W=\frac{1}{w}.$ Then, $W(x)=g^{\ast}\left(  \frac{1}{k+1}B\left(
\frac{k+1}{2},\frac{1}{2}\right)  \left(  x^{2}+y^{2}\right)  ^{k/2+1}\right)
$ and $W^{\ast}(t)=g^{\ast}(t)$ (cf. \cite{Talen}). It follows that $W\in
M(\Phi)\left(  \mathbb{R}^{n}\right)  $.
\end{proof}

\subsection{The unit sphere}

Let $n\geq2$ be an integer, and let $\mathbb{S}^{n}\subset\mathbb{R}^{n+1}$ be
the unit sphere. For each $n\geq2,$ the $n-$dimensional Hausdorff measure of
$\mathbb{S}^{n}$ is given by $\omega_{n}=2\pi^{\frac{n+1}{2}}/\Gamma
(\frac{n+1}{2}).$ On $\mathbb{S}^{n}$ we consider the geodesic distance $d,$
and the uniform probability measure $\sigma_{n}$. For $\theta\in\lbrack
-\pi/2,\pi/2]$, let%
\[
\varphi_{n}(\theta)=\frac{\omega_{n-1}}{\omega_{n}}\cos^{n-1}\theta
\text{\ \ \ \ and \ \ \ \ \ }\Phi_{n}(\theta)=\int_{-\pi/2}^{\theta}%
\varphi_{n}(s)ds.
\]
It is known that the isoperimetric profile of the sphere $I_{\mathbb{S}^{n}}$
coincides with $I_{n}=\varphi_{n}\circ\Phi_{n}^{-1}$ $\ $(cf. \cite{bart})$.$

\begin{theorem}
\label{esferico}Let $g:[0,1]\rightarrow\lbrack0,\mathbb{\infty)}$ be such that
$g>0$ a.e., and
\[
\sup_{0<t<1}g^{\ast}(t)\frac{\min\{t,1/2\}}{I_{n}(\min\{t,1/2\})}<\infty,
\]
where the rearrangement is taken with respect to the Lebesgue measure on
$[0,1].$ Let $w:\mathbb{S}^{n}\rightarrow\mathbb{R}^{+},$ be defined by
\[
w(\theta_{1},.....\theta_{n+1})=\frac{1}{g^{\ast}(\Phi_{n}(\theta_{1}))}.
\]
Then, $w$ is an isoperimetric weight.
\end{theorem}

\begin{proof}
Let
\[
W(\theta)=\frac{1}{w(\theta_{1},.....,\theta_{n+1})}:=g^{\ast}(\Phi_{n}%
(\theta_{1})),\text{ \ \ \ }(\theta_{1},.....\theta_{n+1})\in\mathbb{S}^{n}.
\]
We need to show that $W=\frac{1}{w}\in M(\Phi)\left(  \mathbb{S}^{n}%
,\mu\right)  .$ Let $m_{f}(t)$ denote the distribution function of $f$ with
respect to the Lebesgue measure on $[0,1],$ and let $\mu:=\sigma_{n}.$ Then,
$W$ and $g$ are equimeasurable. Indeed,%
\begin{align*}
\mu_{W}(t)  &  =\mu\{\theta\in\mathbb{S}^{n}:W(x)>t\}\\
&  =\mu\{\theta\in\mathbb{S}^{n}:g^{\ast}(\Phi_{n}(\theta_{1}))>t\}\\
&  =\mu\{\theta\in\mathbb{S}^{n}:\Phi_{n}(\theta_{1})\leq m_{g}(t)\}\\
&  =\mu\{\theta\in\mathbb{S}^{n}:\theta_{1}\leq\Phi_{n}^{-1}(m_{g}(t))\}\\
&  =m_{g}(t).
\end{align*}
Therefore,%
\[
W_{\mu}^{\ast}(t)=g^{\ast}(t).
\]
Consequently, in view of our assumptions on $g,$ $W=\frac{1}{w}\in
M(\Phi)\left(  \mathbb{S}^{n},\mu\right)  .$
\end{proof}

\subsection{Log concave measures\label{logconcave}}

We consider product measures on $\mathbb{R}^{n}$ that are constructed using
the measures on $\mathbb{R}$ defined by the densities
\[
\mathbf{d}\mu_{\Psi}(x)=Z_{\Psi}^{-1}\exp\left(  -\Psi(\left\vert x\right\vert
)\right)  dx=\varphi(x)dx,\text{ \ \ \ }x\in\mathbb{R},
\]
where $\Psi$ is convex, $\sqrt{\Psi}$ concave, $\Psi(0)=0,$ and such that
$\Psi$ is $\mathcal{C}^{2}$ on $[\Psi^{-1}(1),+\infty),$ and where, moreover,
$Z_{\Psi}^{-1}$ is chosen to ensure that $\mu_{\Psi}(\mathbb{R)=}1$.

Let $H:\mathbb{R}\rightarrow(0,1)$ be the distribution function of $\mu_{\Psi
}$, i.e.
\begin{equation}
H(r)=\int_{-\infty}^{r}\varphi(x)dx=\mu_{\Psi}(-\infty,r). \label{def}%
\end{equation}
It is known that the isoperimetric profile for $(\mathbb{R},d_{n},\mu_{\Psi
})\mathbb{\ }$is given by (cf. \cite{Bor} and \cite{Bob})
\[
I_{\mu_{\Psi}}(t)=\varphi\left(  H^{-1}(\min\{t,1-t\}\right)  =\varphi\left(
H^{-1}(t\right)  ),\text{ \ \ \ }t\in\lbrack0,1].
\]
We shall denote by $\mu_{\Psi}^{\otimes n}=$ $\underset{n\text{ times}%
}{\underbrace{\mu_{\Psi}\otimes\mu_{\Psi}\otimes...\otimes\mu_{\Psi}}},$ the
product probability measures on $\mathbb{R}^{n}.$ It is known that the
isoperimetric profiles $I_{\mu_{\Psi}^{\otimes n}}$ are dimension free (cf.
\cite{barthe}): there exists a constant $c(\Psi)$ such that for all
$n\in\mathbb{N}$%
\[
c(\Psi)I_{\mu_{\Psi}}(t)\leq I_{\mu_{\Psi}^{\otimes n}}(t)\leq I_{\mu_{\Psi}%
}(t).
\]

\begin{theorem}
Let $g:[0,1]\rightarrow\lbrack0,\mathbb{\infty)}$ be such that $g>0$ a.e.,
and
\begin{equation}
\sup_{0<t<1}g^{\ast}(t)\frac{\min\{t,1/2)\}}{I_{\mu_{\Psi}}(\min
\{t,1/2\})}<\infty, \label{xfin}%
\end{equation}
where the rearrangement is taken with respect to the Lebesgue measure on
$[0,1].$ Let $w:\mathbb{R}^{n}\rightarrow\mathbb{R}^{+}$ be defined by,
\[
w(x_{1},.....x_{n})=\frac{1}{g^{\ast}(H(x_{1}))}.
\]
Then, $w$ is an isoperimetric weight.
\end{theorem}

\begin{proof}
Let
\[
W(x)=\frac{1}{w(x_{1},.....x_{n})}:=g^{\ast}(H(x_{1})),\text{ \ \ \ }%
x\in\mathbb{R}^{n}.
\]
A calculation, similar to the one given during the course of the proof of
Theorem \ref{esferico}, shows that $W$ and $g$ are equimeasurable. Thus,%
\[
W_{\mu}^{\ast}(t)=g^{\ast}(t),
\]
and we see that $W\in M(\Phi)\left(  \mathbb{R}^{n}\right)  ,$ as we wished to show.
\end{proof}

\begin{example}
The prototype function $g$ that satisfies (\ref{xfin}) is given by
$g(t)=\frac{I_{\mu_{\Psi}}(\min\{t,1/2\})}{\min\{t,1/2\}}.$ In fact, since
$\frac{I_{\mu_{\Psi}}(t)}{t}$ decreases, we see that $g^{\ast}(t)=g(t).$ In
particular, for $\Psi(\left\vert x\right\vert )=\frac{\left\vert x\right\vert
^{p}}{p},$ $p\in\lbrack1,2],$ the isoperimetric profile $I_{\mu_{\Psi}}(t)$
satisfies (cf. \cite{Talen})%
\[
I_{\mu_{\Psi}}(t)\simeq t\left(  \log\frac{1}{t}\right)  ^{1-1/p},\text{
\ }0<t\leq1/2.
\]
Thus, if we let $g:[0,1]\rightarrow\lbrack0,\mathbb{\infty)}$ be such that
$g>0$ a.e., and suppose that
\[
\sup_{0<t<1/2}g^{\ast}(t)\frac{1}{\left(  \log\frac{1}{t}\right)  ^{1-1/p}%
}<\infty,
\]
it follows that the function%
\[
W(x_{1},.....x_{n})=\frac{1}{w(x_{1},.....x_{n})}:=g^{\ast}\left(  Z_{\Psi
}^{-1}\int_{-\infty}^{x_{1}}\exp\left(  \frac{-\left\vert x_{1}\right\vert
^{p}}{p}\right)  dx_{1}\right)
\]
is an isoperimetric weight.
\end{example}

\subsection{Transference\label{secc:transfer}}

We indicate a simple transference result that follows directly from the
characterization of isoperimetric weights in terms of Marcinkiewicz spaces.
Suppose that $(\Omega_{i},\mu_{i},d_{i}),i=1,2,$ are two metric measure spaces
as above, such that, moreover with $\mu_{1}(\Omega_{1})=\mu_{2}(\Omega_{2}),$
and $I_{1}:=I_{(\Omega_{1},\mu_{1},d_{1})}\geq I_{2}:=I_{(\Omega_{2},\mu
_{2},d_{2})}.$ Then, the corresponding uncertainty inequalities for
$(\Omega_{2},\mu_{2},d_{2})$ can be transferred to $(\Omega_{1},\mu_{1}%
,d_{1})$ in the sense that, for all $f\in Lip(\Omega_{1})$ that satisfy
suitably prescribed cancellations\footnote{The key point of this transfer is
that, by abuse of notation, we have \textquotedblleft
switched\textquotedblright\ the isoperimetric weights of $(\Omega_{1},\mu
_{1})$ by the \textquotedblleft isoperimetric weights\textquotedblright\ of
$(\Omega_{2},\mu_{2}).$ In this sense the transfer is apparently connected
with the construction of representations of the space $M(\Phi)$ for different
metric measure spaces. It would be interesting to study the connection of our
transference result with the recent results on the transport of weighted
Poincar\'{e} inequalities (cf. \cite{dariogoz}).},%
\begin{equation}
\left\Vert f\right\Vert _{L^{1}(\Omega_{1},\mu_{1})}\leq C\left\Vert
w\right\Vert _{M(\Phi_{2})}^{1/2}\left\Vert \left\vert \nabla f\right\vert
\right\Vert _{L^{1}(\Omega_{1},\mu_{1})}^{1/2}\left\Vert wf\right\Vert
_{L^{1}(\Omega_{1},\mu_{1})}^{1/2},w\in M(\Phi_{2})(\Omega_{1},\mu_{1}),
\label{uncertau}%
\end{equation}
where $\Phi_{i}(t)=\frac{t}{I_{i}(t)}.$

Indeed, if $I_{1}\geq I_{2},$ then $M(\Phi_{2})\subset M(\Phi_{1}).$
Therefore, the result follows from Lemma \ref{lmar} and Remark \ref{reml1}.

\subsection{Strichartz inequalities and Sobolev inequalities\label{seccstrich}%
}

The initial step of the interpolation process that leads to uncertainty
inequalities is directly connected with the following inequality that one
finds in Strichartz \cite{strichartz}. A different proof with sharp constants
\footnote{Here we shall not be concerned with the best value of the constants
involved but we do note that $c_{n}(p)\rightarrow\infty$ as $p\rightarrow n.$}
was later found by Faris \cite{faris}%
\[
\left\Vert fg\right\Vert _{L^{p}(\mathbb{R}^{n})}\leq c_{n}(p)\left\Vert
g\right\Vert _{L^{n,\infty}}\left\Vert \left\vert \nabla f\right\vert
\right\Vert _{L^{p}(\mathbb{R}^{n})},\text{ }f\in C_{0}^{\infty}%
(\mathbb{R}^{n}),\text{ }1\leq p<n.
\]
These inequalities are connected with the classical Sobolev inequality via the
sharp form of the Sobolev inequality involving Lorentz spaces (cf. \cite{tal}
and the references therein),%
\begin{equation}
\left\{  \int_{0}^{\infty}(f^{\ast}(s)s^{1/\bar{p}})^{p}\frac{ds}{s}\right\}
^{1/p}\leq C_{n}(p)\left\Vert \left\vert \nabla f\right\vert \right\Vert
_{L^{p}(\mathbb{R}^{n})},\text{ }f\in C_{0}^{\infty}(\mathbb{R}^{n}),\text{
}1\leq p<n,\text{ }\frac{1}{\bar{p}}=\frac{1}{p}-\frac{1}{n}. \label{str1}%
\end{equation}
Indeed, suppose that $f\in C_{0}^{\infty}(\mathbb{R}^{n}),g\in L^{n,\infty},$
and let $1\leq p<n;$ then%
\begin{align*}
\left\Vert fg\right\Vert _{L^{p}(\mathbb{R}^{n})}^{p}  &  \leq\int_{0}%
^{\infty}f^{\ast}(s)^{p}g^{\ast}(s)^{p}ds\\
&  \leq\left\Vert g\right\Vert _{L^{n,\infty}}^{p}\int_{0}^{\infty}f^{\ast
}(s)^{p}s^{-p/n}ds\\
&  =\left\Vert g\right\Vert _{L^{n,\infty}}^{p}\int_{0}^{\infty}f^{\ast
}(s)^{p}s^{p/\bar{p}}\frac{ds}{s}\\
&  \leq C_{n}(p)^{p}\left\Vert g\right\Vert _{L^{n,\infty}}^{p}\left\Vert
\left\vert \nabla f\right\vert \right\Vert _{L^{p}(\mathbb{R}^{n})}^{p}.
\end{align*}
Faris' method is closely connected with the above presentation. Note that in
the context of (\ref{str1}) the class of isoperimetric weights can be
described as
\[
\{w:\frac{1}{w}\in M(\Phi)\}=\{w:\frac{1}{w}\in L^{n,\infty}\}=\{w:\frac{1}%
{w}\text{ is a Strichartz multiplier}\}.
\]
When $p\rightarrow n$ the constant $C_{n}(p)$ blows up. The sharp end point
result for $p=n$ is provided by the Brezis-Wainger inequality. Let $\Omega$ be
a domain in $\mathbb{R}^{n}.$ Then, for all functions $f\in C_{0}^{\infty
}(\Omega)$,%
\[
\left\{  \int_{0}^{\left\vert \Omega\right\vert }\frac{f^{\ast}(s)^{n}%
}{(1+\log\frac{\left\vert \Omega\right\vert }{s})^{n}}\frac{ds}{s}\right\}
^{1/n}\leq c_{n}\left\Vert \left\vert \nabla f\right\vert \right\Vert
_{L^{n}(\mathbb{R}^{n})}.
\]
Thus, in this case, the corresponding Strichartz inequality holds if we
replace the condition \textquotedblleft$g\in L^{n,\infty}(\Omega)"$ by
\textquotedblleft$g\in L_{\log}(n,\infty)(\Omega)",$ where
\[
\left\Vert g\right\Vert _{L_{\log}(n,\infty)(\Omega)}=\sup_{t}\left\{
g^{\ast}(t)t^{1/n}(1+\log\frac{\left\vert \Omega\right\vert }{t})\right\}  .
\]
In this notation we have,%
\[
\left\Vert fg\right\Vert _{L^{n}(\mathbb{R}^{n})}^{n}\leq C_{n}(n)^{n}%
\left\Vert g\right\Vert _{L_{\log}(n,\infty)(\Omega)}^{n}\left\Vert \left\vert
\nabla f\right\vert \right\Vert _{L^{n}(\mathbb{R}^{n})}^{n}.
\]

More generally, let us consider Sobolev inequalities on a metric measure space
$\left(  \Omega,\mu,d\right)  $ using rearrangement invariant norms. Let $I$
be an isoperimetric estimator for $\left(  \Omega,\mu,d\right)  $ and on
measurable functions on $(0,\mu(\Omega))$ let us define the isoperimetric
Hardy operator $\tilde{Q}_{I}$ by
\[
\tilde{Q}_{I}f(t)=\frac{I(t)}{t}\int_{t}^{\mu(\Omega)/2}f(s)\frac{ds}{I(s)}.
\]
We give the details for the case $\mu(\Omega)=\infty,$ the case $\mu
(\Omega)<\infty$ follows mutatis mutandi. From Theorem \ref{teopp}, part 2,
below,%
\[
\left\Vert f_{\mu}^{\ast}(t)\frac{I(t)}{t}\right\Vert _{\bar{X}}\leq\left\Vert
\tilde{Q}_{I}\right\Vert _{\bar{X}\rightarrow\bar{X}}\left\Vert \left\vert
\nabla f\right\vert \right\Vert _{X},\text{ \ \ \ }f\in Lip_{0}(\Omega).
\]
We can then reinterpret the last inequality as a weighted norm inequality
(\textquotedblleft the Strichartz inequality in $X"):$ for $f\in
Lip_{0}(\Omega),$ and $g\in M(\Phi),$
\begin{align*}
\left\Vert fg\right\Vert _{X}  &  \leq\left\Vert f_{\mu}^{\ast}(t)g_{\mu
}^{\ast}\right\Vert _{\bar{X}}\\
&  \leq\left\Vert g\right\Vert _{M(\Phi)}\left\Vert f_{\mu}^{\ast}%
(t)\frac{I(t)}{t}\right\Vert _{\bar{X}}\\
&  \leq\left\Vert g\right\Vert _{M(\Phi)}\left\Vert \tilde{Q}_{I}\right\Vert
_{\bar{X}\rightarrow\bar{X}}\left\Vert \left\vert \nabla f\right\vert
\right\Vert _{X}.
\end{align*}

\subsection{Besov Inequalities\label{seccbesov}}

In this brief section we indicate how the inequalities can be extended to
suitable Besov spaces. Here we are aiming to illustrate the method rather than
to prove the most general results. Thus, we shall focus on $(\Omega
,d,\mu)=\mathbb{R}^{n},$ and $L^{q}$ spaces.

Our starting point is the following equivalence (cf. \cite{bl}) which here we
may take as a definition: Let $1\leq q\leq\infty,0<s<1,$
\[
\left\Vert \frac{t^{-s/n}\omega_{q}(t^{1/n},f)}{t^{1/q}}\right\Vert
_{L^{q}(0,\infty)}\approx\left\Vert f\right\Vert _{\mathring{B}_{q,q}%
^{s}(\mathbb{R}^{n})},
\]
where $\omega_{q}$ is the $q-$modulus of continuity defined by%
\[
\omega_{q}(t,f)=\sup_{\left\vert h\right\vert \leq t}\left\Vert f(\circ
+h)-f(\circ)\right\Vert _{L^{q}\left(  \mathbb{R}^{n}\right)  }.
\]
Suppose that $f^{\ast\ast}(\infty)=0,$ then following estimate is well known
(cf. \cite{BS}, \cite{mamiproc} and the references therein)
\begin{equation}
f^{\ast\ast}(t)\leq c\int_{t}^{\infty}\frac{\omega_{q}(s^{1/n},f)}{s^{1/q}%
}\frac{ds}{s}. \label{ayer}%
\end{equation}

Let $w$ be an isoperimetric weight (i.e. $\frac{1}{w}\in L^{n,\infty}).$ Let
$\alpha>0$, $1\leq q<\infty$, $0<\theta<1,$ with $\theta<n/q.$ Then the
following estimate holds,
\[
\left\Vert f\right\Vert _{L^{q}}\leq cr\left\Vert f\right\Vert _{\mathring
{B}_{q,q}^{\theta}(\mathbb{R}^{n})}+r^{-\alpha}\left\Vert w^{\theta\alpha
}f\right\Vert _{L^{q}},
\]

where $c>0$ is an absolute constant. Indeed, following a familiar argument we
have
\begin{align*}
\left\Vert f\right\Vert _{L^{q}}  &  =\left\Vert f\left(  \frac{w}{w}\right)
^{\theta}\right\Vert _{L^{q}}\leq\left\Vert f\left(  \frac{w}{w}\right)
^{\theta}\chi_{\left\{  w\leq r^{1/\theta}\right\}  }\right\Vert _{L^{q}%
}+\left\Vert f\left(  \frac{w}{w}\right)  ^{\theta}\chi_{\left\{
w>r^{1/\theta}\right\}  }\right\Vert _{L^{q}}\\
&  \leq r\left\Vert f\left(  \frac{1}{w}\right)  ^{\theta}\right\Vert _{L^{q}%
}+\left\Vert f\left(  \frac{w}{w}\right)  ^{\theta\alpha}\chi_{\left\{
w>r^{1/\theta}\right\}  }\right\Vert _{L^{q}}\\
&  \leq r\left\Vert f\left(  \frac{1}{w}\right)  ^{\theta}\right\Vert _{L^{q}%
}+r^{-\alpha}\left\Vert w^{\alpha\theta}f\right\Vert _{L^{q}}.
\end{align*}
It remains to estimate the first term,
\begin{align*}
\left\Vert f\left(  \frac{1}{w}\right)  ^{\theta}\right\Vert _{L^{q}}  &
\leq\left\Vert f^{\ast}(t)\left(  \left(  \frac{1}{w}\right)  ^{\ast
}(t)\right)  ^{\theta}t^{\theta/n}t^{-\theta/n}\right\Vert _{L^{q}}\\
&  \leq\left\Vert \frac{1}{w}\right\Vert _{L^{n,\infty}}^{\theta}\left\Vert
f^{\ast}(t)t^{-\theta/n}\right\Vert _{L^{q}}\\
&  \leq\left\Vert \frac{1}{w}\right\Vert _{L^{n,\infty}}^{\theta}\left\Vert
f^{\ast\ast}(t)t^{-\theta/n}\right\Vert _{L^{q}}\\
&  \leq c\left\Vert \frac{1}{w}\right\Vert _{L^{n,\infty}}^{\theta}\left\Vert
t^{-\theta/n}\int_{t}^{\infty}\frac{\omega_{q}(s^{1/n},f)}{s^{1/q}}\frac
{ds}{s}\right\Vert _{L^{q}}\text{ \ \ (by (\ref{ayer}))}\\
&  =c\left\Vert \frac{1}{w}\right\Vert _{L^{n,\infty}}^{\theta}\left\Vert
t^{-\theta/n}\int_{t}^{\infty}\left(  \frac{s^{-\theta/n}\omega_{q}%
(s^{1/n},f)}{s^{1/q}}\right)  s^{\theta/n}\frac{ds}{s}\right\Vert _{L^{q}}\\
&  \leq C\left\Vert \frac{1}{w}\right\Vert _{L^{n,\infty}}^{\theta}\left\Vert
\frac{s^{-\theta/n}\omega_{q}(s^{1/n},f)}{s^{1/q}}\right\Vert _{L^{q}}\text{
(since }\theta<n/q)\\
&  =C\left\Vert \frac{1}{w}\right\Vert _{L^{n,\infty}}^{\theta}\left\Vert
f\right\Vert _{\mathring{B}_{q,q}^{\theta}(\mathbb{R}^{n})}.
\end{align*}

Using the inequalities of \cite{mamiaster} it is possible to extend these
results to Besov spaces on metric spaces but the development is too long and
technical, and falls outside the scope of the present paper.

\section{Rearrangement Invariant Uncertainty Inequalities\label{unrii}}

In this section we obtain uncertainty inequalities modeled on (\ref{unl1}),
where $L^{1}$ is replaced by a suitable r.i. space.

Our approach is based on the following Sobolev inequalities for r.i. spaces
(cf. \cite{mamiadv} where results of this type were obtained with more
restrictions on the ambient measure space).

\begin{theorem}
\label{teopp}Let $X$ be an r.i. space on $\Omega$ such that $\tilde{Q}$ is
bounded on $\bar{X}.$ The following statements hold

\begin{enumerate}
\item Suppose that $\mu(\Omega)<\infty.$ Then,
\begin{equation}
c_{X,I}=\left\Vert \frac{I(t)}{t}\chi_{(0,\mu(\Omega)/2)}(t)\right\Vert
_{\bar{X}}<\infty, \label{cubo}%
\end{equation}
and for all bounded functions\ $f\in Lip(\Omega),\ $we have
\begin{equation}
\left\Vert f_{\mu}^{\ast}(t)\frac{I(t)}{t}\chi_{(0,\mu(\Omega)/2)}%
(t)\right\Vert _{\bar{X}}\leq\left\Vert \tilde{Q}_{I}\right\Vert _{\bar
{X}\rightarrow\bar{X}}\left\Vert \left\vert \nabla f\right\vert \right\Vert
_{X}+\frac{2c_{X,I}}{\mu(\Omega)}\int_{\Omega}\left\vert f(x)\right\vert d\mu.
\label{salvo}%
\end{equation}

\item If $\mu(\Omega)=\infty,$ then
\begin{equation}
\left\Vert f_{\mu}^{\ast}(t)\frac{I(t)}{t}\right\Vert _{\bar{X}}\leq\left\Vert
\tilde{Q}_{I}\right\Vert _{\bar{X}\rightarrow\bar{X}}\left\Vert \left\vert
\nabla f\right\vert \right\Vert _{X},\text{ \ \ \ }f\in Lip_{0}(\Omega).
\label{salvo1}%
\end{equation}

\end{enumerate}
\end{theorem}

The proof of this theorem will be given at the end of this section. First we
consider the corresponding uncertainty inequalities.

\begin{theorem}
Let $X$ be an r.i. space on $\Omega$ such that $\tilde{Q}$ is bounded on
$\bar{X}.$ Let $w$ be an isoperimetric weight and let $\alpha>0,$ Then, there
exists a constant $C=C(w,X)$ such that
\[
\left\Vert f\right\Vert _{X}\leq rC\left\Vert \left\vert \nabla f\right\vert
\right\Vert _{X}+r^{-\alpha}\left\Vert w^{\alpha}f\right\Vert _{X},\text{ for
all }r>0,
\]
and for all bounded $f\in Lip(\Omega)$ such that $m(f)=0$ or $\int_{\Omega
}fd\mu=0$ if $\mu(\Omega)<\infty$ (or for all $f\in Lip_{0}(\Omega)$ if
$\mu(\Omega)=\infty).$
\end{theorem}

\begin{proof}
a. Suppose that $\mu(\Omega)<\infty.$ Let $f\in Lip(\Omega),$ and let $w$ be
an isoperimetric weight. Then, for all $\alpha>0,$ we have%
\begin{align}
\left\Vert f\right\Vert _{X}  &  =\left\Vert f\frac{w}{w}\right\Vert _{X}%
\leq\left\Vert f\frac{w}{w}\chi_{\left\{  w\leq r\right\}  }\right\Vert
_{X}+\left\Vert f\frac{w}{w}\chi_{\left\{  w>r\right\}  }\right\Vert
_{X}\label{primerpaso}\\
&  \leq r\left\Vert f\frac{1}{w}\right\Vert _{X}+r^{-\alpha}\left\Vert
w^{\alpha}f\right\Vert _{X}.\nonumber
\end{align}
To estimate the first term in (\ref{primerpaso}), we write
\begin{align*}
\left\Vert f\frac{1}{w}\right\Vert _{X}  &  =\left\Vert \left(  f\frac{1}%
{w}\right)  _{\mu}^{\ast}\right\Vert _{\bar{X}}\\
&  \leq\left\Vert f_{\mu}^{\ast}(s)\left(  \frac{1}{w}\right)  _{\mu}^{\ast
}(s)\right\Vert _{\bar{X}}\text{ \ (by (\ref{a3}) and (\ref{hardyine}))}\\
&  \leq2\left\Vert f_{\mu}^{\ast}(s)\left(  \frac{1}{w}\right)  _{\mu}^{\ast
}(s)\chi_{(0,\mu(\Omega)/2)}(s)\right\Vert _{\bar{X}}\\
&  =2\left\Vert f_{\mu}^{\ast}(s)\left(  \frac{1}{w}\right)  _{\mu}^{\ast
}(s)\frac{s}{I(s)}\frac{I(s)}{s}\chi_{(0,\mu(\Omega)/2)}(s)\right\Vert
_{\bar{X}}\\
&  \leq2\left(  \sup_{0<s<\mu(\Omega)/2}\left(  \frac{1}{w}\right)  _{\mu
}^{\ast}(s)\frac{s}{I(s)}\right)  \left\Vert f_{\mu}^{\ast}(s)\frac{I(s)}%
{s}\chi_{(0,\mu(\Omega)/2)}(s)\right\Vert _{\bar{X}}\\
&  =2\left\Vert W\right\Vert _{M\left(  \Phi\right)  }\left\Vert f_{\mu}%
^{\ast}(s)\frac{I(s)}{s}\chi_{(0,\mu(\Omega)/2)}(s)\right\Vert _{\bar{X}}\\
&  \leq2\left\Vert W\right\Vert _{M\left(  \Phi\right)  }\left(  \left\Vert
\tilde{Q}_{I}\right\Vert _{\bar{X}\rightarrow\bar{X}}\left\Vert \left\vert
\nabla f\right\vert \right\Vert _{X}+\frac{2c_{XI}}{\mu(\Omega)}\int_{\Omega
}\left\vert f\right\vert d\mu\right)  \text{ (by (\ref{salvo})).}%
\end{align*}

Let \ $f\in Lip(\Omega)$ be such that $m(f)=0$ or $\int_{\Omega}fd\mu=0,$ by
Poincar\'{e}'s inequality (cf. Remark \ref{poincar})%
\begin{align*}
\int_{{\Omega}}\left\vert f\right\vert d\mu &  \leq\frac{\mu({\Omega})}%
{I(\mu({\Omega})/2)}\int_{{\Omega}}\left\vert \nabla f\right\vert d\mu\\
&  \leq\frac{\mu({\Omega})}{I(\mu({\Omega})/2)}\left\Vert \left\vert \nabla
f\right\vert \right\Vert _{X}\frac{\mu({\Omega})}{\left\Vert \chi_{\Omega
}\right\Vert _{X}}\text{ (by H\"{o}lder's inequality)}.
\end{align*}
Summarizing,%
\[
\left\Vert f\right\Vert _{X}\leq rC\left\Vert \left\vert \nabla f\right\vert
\right\Vert _{X}+r^{-\alpha}\left\Vert w^{\alpha}f\right\Vert _{X},
\]
where $C=2\left\Vert W\right\Vert _{M\left(  \Phi\right)  }\left\Vert
\tilde{Q}_{I}\right\Vert _{\bar{X}\rightarrow\bar{X}}+\frac{2c_{XI}\mu
({\Omega})}{I(\mu({\Omega})/2)\left\Vert \chi_{\Omega}\right\Vert _{X}}.$

b. $\mu(\Omega)=\infty.$ We follow the same steps as in the previous case, but
now we use (\ref{salvo1}) instead (\ref{salvo}). Notice that the extra
$L^{1}-$term does not appear in this case.
\end{proof}

\subsubsection{The proof of Theorem \ref{teopp}.}

In this section we prove Theorem \ref{teopp}. We need the following Lemma (see
\cite{mamiaster}, \cite{mamiconc}).

\begin{lemma}
\label{ll1}Let $\left(  \Omega,\mu,d\right)  $ be a metric space and let $I$
be an isoperimetric estimator for $\left(  \Omega,\mu,d\right)  .$ Let $h$ be
a bounded Lip function on $\Omega$. Then there exists a sequence of bounded
functions $\left(  h_{n}\right)  _{n}$ $\subset Lip(\Omega)$ , such that

\begin{enumerate}
\item
\begin{equation}
\left|  \nabla h_{n}(x)\right|  \leq(1+\frac{1}{n})\left|  \nabla h(x)\right|
,\text{ \ }x\in\Omega. \label{cota01}%
\end{equation}

\item
\begin{equation}
h_{n}\underset{n\rightarrow0}{\rightarrow}h\text{ in }L^{1}. \label{converge}%
\end{equation}

\item The functions $\left(  h_{n}\right)  _{\mu}^{\ast}$ are locally
absolutely continuous and for any r.i. space $X$ on $\Omega$%
\begin{equation}
\left\Vert \left(  -\left\vert h_{n}\right\vert _{\mu}^{\ast}\right)
^{\prime}(\cdot)I(\cdot)\right\Vert _{\bar{X}}\leq\left\Vert \left\vert \nabla
h_{n}\right\vert \right\Vert _{X},\text{ for all }n\in\mathbb{N}\text{.}
\label{aa}%
\end{equation}

\end{enumerate}
\end{lemma}

\begin{proof}
(of Theorem \ref{teopp})

a. Suppose that $\mu(\Omega)<\infty$. The fact that $c_{X,I}=\left\Vert
\frac{I(t)}{t}\chi_{(0,\mu(\Omega)/2)}(t)\right\Vert _{\bar{X}}<\infty,$
follows easily from the fact that $\tilde{Q}_{I}$ is bounded. Indeed, for
$0<t<\mu(\Omega)/4,$ we have that
\[
\tilde{Q}_{I}\chi_{(0,\mu(\Omega)/2)}(t)=\frac{I(t)}{t}\int_{t}^{\mu
(\Omega)/2}\frac{dr}{I(r)}\geq\frac{I(t)}{t}\int_{\mu(\Omega)/4}^{\mu
(\Omega)/2}\frac{dr}{I(r)},
\]
and (\ref{cubo}) follows.

Let $f$ be a bounded function in $Lip(\Omega)$. Let $\left(  f_{n}\right)
_{n}$ be the sequence associated to $f$ that is provided by Lemma \ref{ll1}.
Since $\left(  f_{n}\right)  _{\mu}^{\ast}$ is locally absolutely continuous,
by the fundamental theorem of calculus we have
\begin{align*}
A(t)  &  =\left(  f_{n}\right)  _{\mu}^{\ast}(t)\frac{I(t)}{t}\chi
_{(0,\mu(\Omega)/2)}(t)\\
&  =\frac{I(t)}{t}\int_{t}^{\mu(\Omega)/2}\left(  -\left(  f_{n}\right)
_{\mu}^{\ast}\right)  ^{\prime}(r)dr+\left(  f_{n}\right)  _{\mu}^{\ast}%
(\mu(\Omega)/2)\frac{I(t)}{t}\chi_{(0,\mu(\Omega)/2)}(t)\\
&  =\frac{I(t)}{t}\int_{t}^{\mu(\Omega)/2}\left(  -\left(  f_{n}\right)
_{\mu}^{\ast}\right)  ^{\prime}(r)I(r)\frac{dr}{I(r)}+\left(  f_{n}\right)
_{\mu}^{\ast}(\mu(\Omega)/2)\frac{I(t)}{t}\chi_{(0,\mu(\Omega)/2)}(t)\\
&  =\tilde{Q}_{I}\left(  \left(  -\left(  f_{n}\right)  _{\mu}^{\ast}\right)
^{\prime}(\cdot)I(\cdot)\right)  (t)+\left(  f_{n}\right)  _{\mu}^{\ast}%
(\mu(\Omega)/2)\frac{I(t)}{t}\chi_{(0,\mu(\Omega)/2)}(t).
\end{align*}
Thus,
\begin{align*}
\left\Vert A(t)\right\Vert _{\bar{X}}  &  \leq\left\Vert \tilde{Q}_{I}\left(
\left(  -\left(  f_{n}\right)  _{\mu}^{\ast}\right)  ^{\prime}(\cdot
)I(\cdot)\right)  (t)\right\Vert _{\bar{X}}+\left(  f_{n}\right)  _{\mu}%
^{\ast}(\mu(\Omega)/2)\left\Vert \frac{I(t)}{t}\chi_{(0,\mu(\Omega
)/2)}(t)\right\Vert _{\bar{X}}\\
&  =I+II.
\end{align*}
Now,
\begin{align*}
I  &  \leq\left\Vert \tilde{Q}_{I}\right\Vert _{\bar{X}\rightarrow\bar{X}%
}\left\Vert \left(  \left(  -\left(  f_{n}\right)  _{\mu}^{\ast}\right)
^{\prime}(\cdot)I(\cdot)\right)  (t)\right\Vert _{\bar{X}}\leq\left\Vert
\tilde{Q}_{I}\right\Vert _{\bar{X}\rightarrow\bar{X}}\left\Vert \left\vert
\nabla h_{n}\right\vert \right\Vert _{X}\text{ \ (by (\ref{aa}))}\\
&  \leq\left\Vert \tilde{Q}_{I}\right\Vert _{\bar{X}\rightarrow\bar{X}%
}(1+\frac{1}{n})\left\Vert \left\vert \nabla f\right\vert \right\Vert
_{X}\text{ \ (by (\ref{cota01})),}%
\end{align*}
and
\begin{align*}
II  &  \leq\left(  \frac{2}{\mu(\Omega)}\int_{\Omega}\left\vert f_{n}%
\right\vert (x)d\mu\right)  \left(  \left\Vert \frac{I(t)}{t}\chi
_{(0,\mu(\Omega)/2)}(t)\right\Vert _{\bar{X}}\right) \\
&  =\frac{2c}{\mu(\Omega)}\int_{\Omega}\left\vert f_{n}\right\vert (x)d\mu.
\end{align*}
Therefore,
\begin{align*}
\left\Vert f_{\mu}^{\ast}(t)\frac{I(t)}{t}\right\Vert  &  \leq\lim
\inf_{n\rightarrow\infty}\left(  \left\Vert \tilde{Q}_{I}\right\Vert _{\bar
{X}\rightarrow\bar{X}}(1+\frac{1}{n})\left\Vert \left\vert \nabla f\right\vert
\right\Vert _{X}+\frac{2c}{\mu(\Omega)}\int_{\Omega}\left\vert f_{n}%
(x)\right\vert d\mu\right) \\
&  =\left\Vert \tilde{Q}_{I}\right\Vert _{\bar{X}\rightarrow\bar{X}}\left\Vert
\left\vert \nabla f\right\vert \right\Vert _{X}+\frac{2C}{\mu(\Omega)}%
\int_{\Omega}\left\vert f(x)\right\vert d\mu\text{\ \ (by (\ref{converge}))}\\
&  =D\left(  \left\Vert \left\vert \nabla f\right\vert \right\Vert _{X}%
+\frac{2}{\mu(\Omega)}\int_{\Omega}\left\vert f(x)\right\vert d\mu\right)  .
\end{align*}

b. $\mu(\Omega)=\infty.$ The proof follows the same argument. Indeed, if $f\in
Lip_{0}(\Omega),$ then $f$ is bounded. Let $\left(  f_{n}\right)  _{n}$ the
sequence associated to $f$ that is provided by Lemma \ref{ll1}. Note that
$\left(  f_{n}\right)  _{\mu}^{\ast}$ is locally absolutely continuous, and
$\left(  f_{n}\right)  _{\mu}^{\ast\ast}(\infty)=0.$ Using the fundamental
theorem of calculus we find
\begin{align*}
\left(  f_{n}\right)  _{\mu}^{\ast}(t)\frac{I(t)}{t}  &  =\frac{I(t)}{t}%
\int_{t}^{\infty}\left(  -\left(  f_{n}\right)  _{\mu}^{\ast}\right)
^{\prime}(r)dr\\
&  =\tilde{Q}_{I}\left(  \left(  -\left(  f_{n}\right)  _{\mu}^{\ast}\right)
^{\prime}(\cdot)I(\cdot)\right)  (t),
\end{align*}
and we conclude the proof as in the previous case.
\end{proof}

\subsection{Final Remarks}

a. We should mention that in the literature one can find $L^{1}$ uncertainty
type inequalities that are not directly related to those treated in our paper.
For example, in \cite{lamo}, sharp constants are obtained for inequalities of
the following type (here $\Omega=\mathbb{R)}$,
\[
\left\Vert f\right\Vert _{1}\left\Vert f\right\Vert _{2}^{2}\leq c\left\Vert
\xi\hat{f}\right\Vert _{2}^{2}\left\Vert x^{2}f\right\Vert _{1}%
\]
or%
\[
\left\Vert f\right\Vert _{2}\leq c\left\Vert x^{2}f\right\Vert _{1}%
^{2/7}\left\Vert \xi\hat{f}\right\Vert _{2}^{5/7}.
\]

b. Multiplier inequalities have a long history. We mention two somewhat
related directions of inquiry that we find intriguing. The multiplier
inequalities exemplified by \cite{sickel}, and the long list of references
therein, and the potential spaces of radial functions exemplified by
\cite{napoli} and the references therein.

c. In this paper we have not considered discrete inequalities. We hope to
discuss the discrete world elsewhere.

\begin{acknowledgement}
Thanks are due to Michael Cwikel for an opportune e-mail that sparked this
research and, at a later stage, for a number of helpful suggestions to improve
the presentation. We are also grateful to the referee for several very useful
suggestions to improve the presentation. After this paper was completed we
learned from Stefan Steinerberger of his interesting related work on
Poincar\'{e} inequalities in the Euclidean setting (cf. \cite{stefan1}).
Moreover, in \cite{stefan1} one can also find a large set of references that
go all the way back to Yau \cite{yau}. In \cite{stefan}, Steinerberger also
studies uncertainty inequalities, but we should note that \cite{stefan1} and
\cite{stefan1} are independent of each other.
\end{acknowledgement}


\begin{thebibliography}{99}                                                                                               %


\bibitem {bart}F. Barthe, \textsl{ Log-concave and spherical models in
isoperimetry}, Geom. Funct. Anal. \textbf{12} (2002), 32-55.

\bibitem {barthe}F. Barthe, P. Cattiaux and C. Roberto, \textsl{Isoperimetry
between exponential and Gaussian}, Electronic Journal of Probability
\textbf{12} (2007), 1212--1237.

\bibitem {bk}W. Beckner, \textsl{A generalized Poincar\'{e} inequality for
Gaussian measures}, Proc. Amer. Math. Soc. \textbf{105} (1989), 397-400.

\bibitem {bk1}W. Beckner, \textsl{Pitt's inequality with sharp convolution
estimates}, Proc. Amer. Math. Soc. \textbf{136} (2008), 1871--1885.

\bibitem {BS}C. Bennett and R. Sharpley, \textsl{Interpolation of Operators},
Academic Press, Boston\textbf{, }1988.

\bibitem {Bob}S. G. Bobkov, \textsl{Extremal properties of half-spaces for
log-concave distributions}, Ann. Probab. \textbf{24} (1996), 35-48.

\bibitem {bobhou}S. G. Bobkov and C. Houdr\'{e}, \textsl{Some connections
between isoperimetric and Sobolev-type inequalities}, Mem. Amer. Math. Soc.
\textbf{129} (1997).

\bibitem {Bor}C. Borell, \textsl{The Brunn-Minkowski inequality in Gauss
space}, Invent. Math. \textbf{30} (1975), 207-216.

\bibitem {ciattia}P. Ciatti, M. G. Cowling and F. Ricci, \textsl{Hardy and
uncertainty inequalities on stratified Lie groups}, Adv. Math. \textbf{277}
(2015), 365--387.

\bibitem {ciatti}P. Ciatti, F. Ricci and M. Sundari,\textsl{
Heisenberg-Pauli-Weyl uncertainty inequalities and polynomial volume growth},
Adv. Math. \textbf{215} (2007), 616--625.

\bibitem {dariogoz}D. Cordero-Erausquin and N. Gozlan, \textsl{Transport
proofs of weighted Poincar\'{e} inequalities for log-concave distributions},
Arxiv Math. (arXiv:1407.3217)

\bibitem {coulh}Th. Coulhon and L. Saloff-Coste, \textsl{Isop\'{e}rim\'{e}trie
pour les groupes et les vari\'{e}t\'{e}s}, Rev. Mat. Iberoamericana \textbf{9}
(1993), 293--314.

\bibitem {daltre}G. M. Dall'ara and D. Trevisan, \textsl{Uncertainty
inequalities on groups and homogeneous spaces via isoperimetric inequalities},
J. Geom. Anal. \textbf{25} (2015), 2262--2283.

\bibitem {faris}W. G. Faris, \textsl{Weak Lebesgue spaces and quantum
mechanical binding}, Duke Math. J.\textbf{ 4} (1976), 365-373.

\bibitem {fe}C. L. Fefferman, \textsl{The uncertainty principle}, Bull. Amer.
Math. Soc. (N.S.) \textbf{9} (1983), 129--206.

\bibitem {folland}G. B. Folland and A. Sitaram, \textsl{The uncertainty
principle: a mathematical survey}, J. Fourier Anal. Appl. \textbf{3} (1997), 207--238.

\bibitem {forzani}L. Forzani, \textsl{Lemas de cubrimiento de tipo Besicovitch
y su aplicaci\'{o}n al estudio del operador maximal de Ornstein-Uhlenbec}k,
Ph.D. Thesis, Univ. Nac. San Luis, Argentina, 1993.

\bibitem {lamo}E. Laeng and C. Morpurgo, \textsl{An uncertainty inequality
involving }$L^{1}$\textsl{ norms}, Proc. Amer. Math. Soc. \textbf{127} (1999), 3565--3572.

\bibitem {mamiproc}J. Martin and M. Milman, \textsl{Symmetrization
inequalities and Sobolev embeddings}, Proc. Amer. Math. Soc. \textbf{134}
(2006), 2335--2347.

\bibitem {mamiadv}J. Martin and M. Milman, \textsl{Pointwise symmetrization
inequalities for Sobolev functions and applications}, Adv. Math. \textbf{225}
(2010), 121-199.

\bibitem {mamiconc}J. Martin and M. Milman, \textsl{Sobolev inequalities,
rearrangements, isoperimetry and interpolation}, Contemporary Math.
\textbf{545} (2011), pp 167-193.

\bibitem {mamiaster}J. Martin and M. Milman, \textsl{Fractional Sobolev
Inequalities: Symmetrization, Isoperimetry and Interpolation}, Asterisqu\'{e}
\textbf{366} (2014).

\bibitem {martini}A. Martini, \textsl{Generalized uncertainty inequalities},
Math. Z. \textbf{265} (2010), 831--848.

\bibitem {mazya}V. G. Maz'ya, \textsl{Sobolev Spaces with applications to
elliptic partial differential equations}. Second, revised and augmented
edition. Grundlehren der Mathematischen Wissenschaften [Fundamental Principles
of Mathematical Sciences], \textbf{342}. Springer, Heidelberg, 2011.

\bibitem {maz1}V. G. Maz'ya, \textsl{The} \textsl{p-conductivity and theorems
on imbedding certain functional spaces into a C-space} (Russian), Dokl. Akad.
Nauk SSSR \textbf{140} (1961), 299--302 (English translation: in Soviet Math.
Dokl. 3 (1962).

\bibitem {napoli}P. L. De Napoli and I. Drelichman, \textsl{Elementary proofs
of embedding theorems for potential spaces of radial functions}, Arxiv. Math (arXiv:1404.7468).

\bibitem {oko}K. A. Okoudjou, L. Saloff-Coste and A. Teplyaev, \textsl{Weak
uncertainty principle for fractals, graphs and metric measure spaces}, Trans.
Amer. Math. Soc. \textbf{360} (2008), 3857--3873.

\bibitem {oneil}R. O'Neil, \textsl{Convolution operators and L(p,q) spaces},
Duke Math. J. \textbf{30} (1963), 129--142.

\bibitem {ricci}F. Ricci, \textsl{Uncertainty inequalities on spaces with
polynomial volume growth}, Rend. Accad. Naz. Sci. XL Mem. Mat. Appl.
\textbf{29} (2005), 327--337.

\bibitem {sal}L. Saloff-Coste, \textsl{Aspects of Sobolev-Type Inequalities},
London Mathematical Society Lecture Note Series \textbf{289}, Cambridge
University Press, Cambridge, 2002.

\bibitem {sickel}W. Sickel,\textsl{ On pointwise multipliers for} $F_{p,q}%
^{s}(\mathbb{R}^{n})$ in case $\sigma_{p,q}<s<n/p,$\ Ann. Mat. Pura App.
\textbf{176} (1999), 209--250.

\bibitem {stewei}E. M. Stein and G. Weiss, \textsl{Fourier analysis on
Euclidean spaces}, Princeton Univ. Press, 1971.

\bibitem {stefan}S. Steinerberger, \textsl{An uncertainty principle on compact
manifolds}, J. Fourier Anal. Appl. \textbf{21} (2015), 575--599.

\bibitem {stefan1}S. Steinerberger, \textsl{Sharp }$L^{1}$\textsl{
Poincar\'{e} inequalities correspond to optimal hypersurface cuts}, Arch.
Math. (Basel) \textbf{105} (2015), 179--188.

\bibitem {strichartz}R. S. Strichartz, \textsl{Multipliers on fractional
Sobolev spaces}, J. Math. Mechanics \textbf{16} (1967), 1031-1060.

\bibitem {tal}G. Talenti,\textsl{ Inequalities in rearrangement-invariant
function spaces}, in: Nonlinear Analysis, Function Spaces and Applications,
vol. 5, Prometheus, Prague, 1995, pp. 177--230.

\bibitem {Talen}G. Talenti, \textsl{A weighted version of a rearrangement
inequality}, Ann. Univ Ferrara \textbf{43} (1997), 121--133.

\bibitem {urbina}W. Urbina\textsl{, An\'{a}lisis Arm\'{o}nico Gausiano: una
visi\'{o}n pan\'{o}ramica}, Bol. Asoc. Mat. Ven.\textbf{V} (1998), 143-181.

\bibitem {weyl}H. Weyl, \textsl{The Theory of Groups and Quantum Mechanics},
Dover, NY 1949.

\bibitem {yau}S.-T. Yau, \textsl{Isoperimetric constants and the first
eigenvalue of a compact Riemannian manifold}, Ann. Sci. Ecole

Norm. Sup. 8 (1975), 487-507.
\end{thebibliography}
\end{document}